\documentclass[11pt]{article}
\usepackage{url}
\usepackage{float}
\usepackage{enumerate}
\usepackage{amsmath}%
\usepackage{stmaryrd}
\usepackage{amsxtra}%
\usepackage{amsfonts}%
\usepackage{amsthm}
\usepackage{amssymb}%
\usepackage{graphicx}
\usepackage{subfigure}
\usepackage[margin=3cm]{geometry}

\newcommand{\iid}{\emph{i.i.d.}}
\newcommand{\K}{{\mathcal K}}
\renewcommand{\bar}[1]{{\overline{#1}}}
\newcommand{\F}{{\mathcal F}}
\newcommand{\X}{{\mathcal X}}
\newcommand{\R}{\mathbb{R}}
\newcommand{\N}{\mathbb{N}}
\newcommand{\Z}{\mathbb{Z}}
\newcommand{\supp}{{\rm supp}}
\newcommand{\eps}{\varepsilon}
\renewcommand{\hat}{\widehat}

\renewcommand{\phi}{\varphi}
\newcommand*\samethanks[1][\value{footnote}]{\footnotemark[#1]}
\newcommand{\vb}[1]{\mathbf{#1}}

\newtheorem{theorem}{Theorem}
\theoremstyle{plain}

\newtheorem{algorithm}{Algorithm}

\newtheorem{lemma}{Lemma}

\newtheorem{proposition}{Proposition}

\numberwithin{equation}{section}

\title{A PDE-based approach to non-dominated sorting\thanks{The research in this paper  was partially supported by NSF grants DMS-0914567 and CCF-1217880, and by  ARO grant W911NF-09-1-0310.}}
\author{Jeff Calder\thanks{Department of Mathematics, University of Michigan. ({\tt \{jcalder,esedoglu\}@umich.edu})}
   \and Selim Esedo\=glu\samethanks
   \and Alfred O. Hero\thanks{Department of Electrical Engineering and Computer Science, University of Michigan. ({\tt hero@eecs.umich.edu})}}

\begin{document} 
\maketitle
\begin{abstract}
Non-dominated sorting is a fundamental combinatorial problem in multiobjective optimization, and is equivalent to the longest chain problem in combinatorics and random growth models for crystals in materials science.  In a previous work~\cite{calder2013}, we showed that non-dominated sorting has a continuum limit that corresponds to solving a Hamilton--Jacobi equation.  In this work we present and analyze a fast numerical scheme for this Hamilton--Jacobi equation, and show how it can be used to design a fast algorithm for approximate non-dominated sorting. 
\end{abstract}
%

\pagestyle{myheadings}
\thispagestyle{plain}
\markboth{Calder et al.}{A PDE-based approach to non-dominated sorting}

\section{Introduction}
Non-dominated sorting is a combinatorial problem that is fundamental in multiobjective optimization, which is ubiquitous is scientific and engineering contexts~\cite{ehrgott2005,deb2001,deb2002}.  The sorting can be viewed as arranging a finite set of points in Euclidean space into layers according to the componentwise partial order.  The layers are obtained by repeated removal of the set of minimal elements.  More formally, given a set $\X_n \subset \R^d$ of $n$ points equipped with the componentwise partial order $\leqq$\footnote{$x\leqq y \iff x_i \leq y_i $ for $i=1,\dots,d$.}, the first layer, often called the first Pareto front and denoted $\F_1$, is the set of minimal elements in $\X_n$.  The second Pareto front $\F_2$ is the set of minimal elements in $\X_n \setminus \F_1$, and in general the $k^{\rm th}$ Pareto front $\F_k$ is given by
\[\F_k = {\rm mimimal \ elements \ of \ } \X_n \setminus \bigcup_{i \leq k-1} \F_i.\]
In the context of multiobjective optimization, the $d$ coordinates of each point in $\X_n$ are the values of the $d$ objective functions evaluated on a given feasible solution.  In this way, each point in $\X_n$ corresponds to a feasible solution and the layers provide an effective ranking of all feasible solutions with respect to the given optimization problem.  Rankings obtained in this way are at the heart of genetic and evolutionary algorithms for multiobjective optimization, which have proven to be valuable tools for finding solutions numerically~\cite{deb2001,deb2002,fonseca1993,fonseca1995,srinivas1994}. Figure \ref{fig:example-fronts} gives a visual illustration of Pareto fronts for randomly generated points.   

\begin{figure}
\centering
\subfigure[$n=50$ points]{\includegraphics[width=0.47\textwidth]{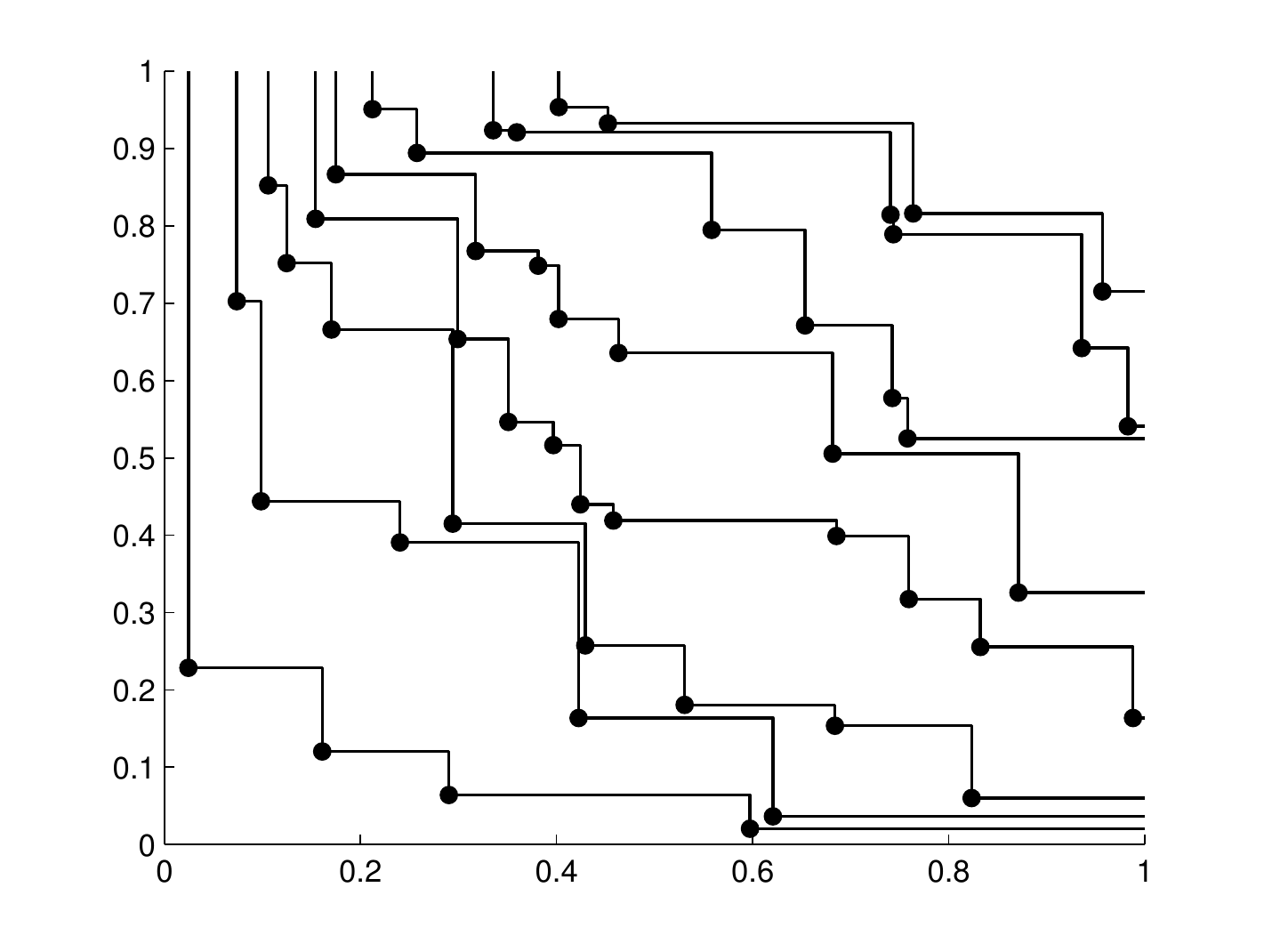}}
\subfigure[$n=10^6$ points]{\includegraphics[width=0.47\textwidth]{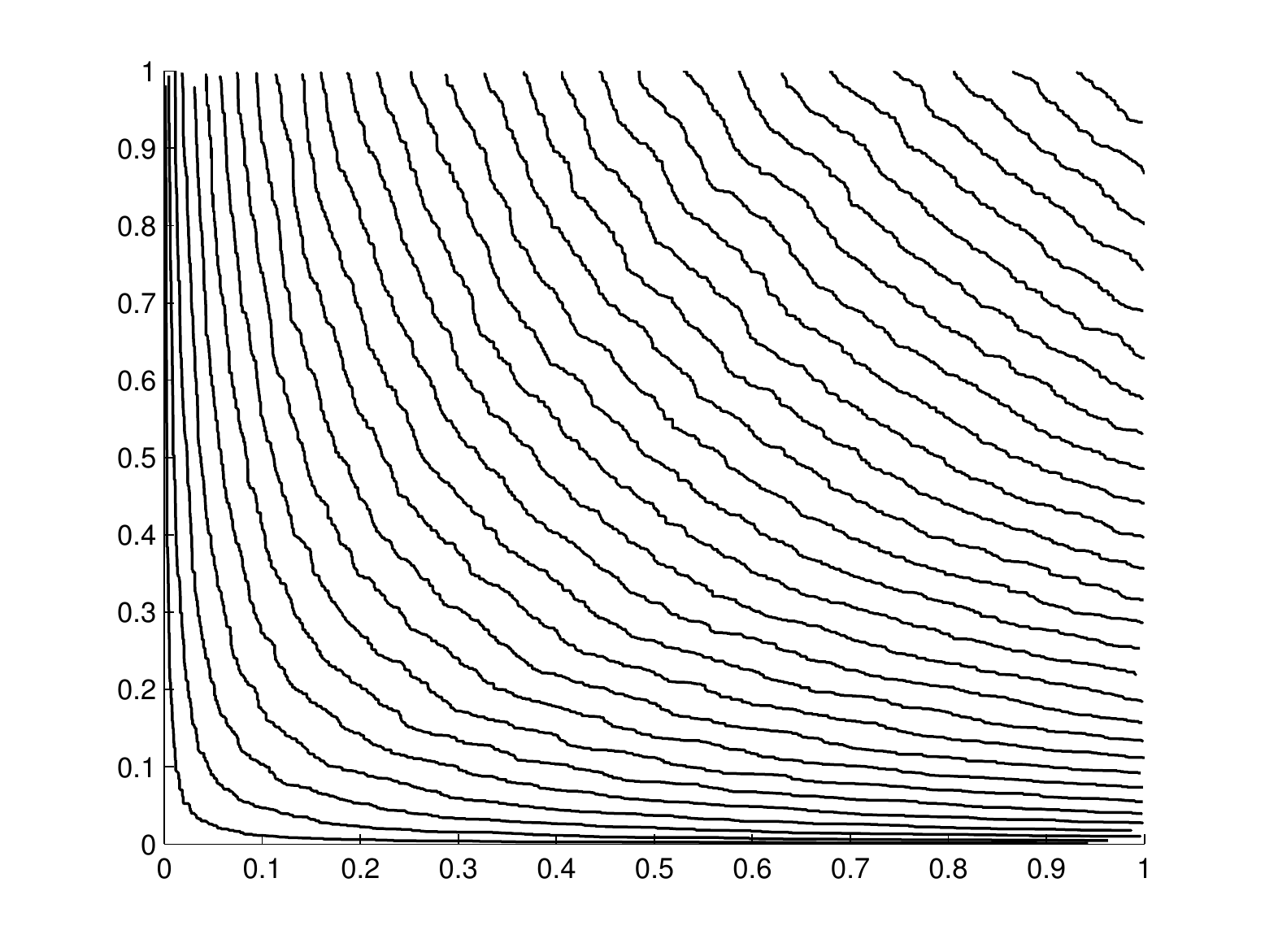}}
\caption{Examples of Pareto fronts for $X_1,\dots X_n$ chosen from the uniform distribution on $[0,1]^2$.  In (b), $29$ equally spaced fronts are depicted out of the $1938$ total fronts.}
\label{fig:example-fronts}
\end{figure}

It is important to note that non-dominated sorting is equivalent to the longest chain problem in combinatorics, which has a long history beginning with Ulam's famous problem of finding the length of a longest increasing subsequence in a sequence of numbers (see~\cite{ulam1961,hammersley1972,bollobas1988,deuschel1995,calder2013} and the references therein).  The longest chain problem is then intimately related to problems in combinatorics and graph theory~\cite{felsner1999,lou1993,viennot1984}, materials science~\cite{prahofer2000}, and molecular biology~\cite{pevzner2000}.
To see this connection, let $u_n(x)$ denote the length of a longest chain\footnote{A \emph{chain} is a totally ordered subset of $\X_n$.} in $\X_n$ consisting of points less than or equal to $x$ with respect to $\leqq$.  If all points in $\X_n$ are distinct, then a point $x \in \X_n$ is a member of $\F_1$ if and only if $u_n(x) = 1$. By peeling off $\F_1$ and making the same argument, we see that $x\in \X_n$ is a member of $\F_2$ if and only if $u_n(x)=2$.  In general, for any $x \in \X_n$ we have
\[x\in \F_k \ \iff \ u_n(x) = k.\]
This is a fundamental observation.  It says that studying the shapes of the Pareto fronts is equivalent to studying the longest chain function $u_n$.  

The longest chain problem has well-understood asymptotics as $n\to \infty$.  In this context, we assume that $\X_n=\{X_1,\dots,X_n\}$ where $X_1,\dots,X_n$ are \iid~random variables in $\R^n$ and let $\ell(n)$ denote the length of a longest chain in $\X_n$.  The seminal work on the problem was done by Hammersley~\cite{hammersley1972}, who studied the problem for $X_1,\dots,X_n$ \iid~uniform on $[0,1]^2$.  He utilized subadditive ergodic theory to show that $n^{-\frac{1}{2}}\ell(n) \to c$ in probability, where $c>0$.  He conjectured that $c=2$, and this was later proven by Vershik and Kerov~\cite{vershik1977} and Logan and Shepp~\cite{logan1977}.  Hammersley's results were generalized to higher dimensions by Bollob\'as and Winkler~\cite{bollobas1988}, who showed that $n^{-\frac{1}{d}}\ell(n) \to c_d$ almost surely, where $0 < c_d < e$ are constants tending to $e$ as $d\to \infty$. The only known values of $c_d$ are $c_1=1$ and $c_2=2$.  Deuschel and Zeitouni~\cite{deuschel1995} provided another generalization of Hammersley's results; for $X_1,\dots,X_n$ \iid~on $[0,1]^2$ with $C^1$ density function $f:[0,1]^2 \to \R$, bounded away from zero, they showed that $n^{-\frac{1}{2}}\ell(n) \to 2\bar{J}$ in probability, where $\bar{J}$ is the supremum of the energy
\[J(\phi) = \int_0^1 \sqrt{\phi'(x) f(x,\phi(x))} \, dx,\]
over all $\phi:[0,1]\to[0,1]$ nondecreasing and right continuous.  

In~\cite{calder2013}, we studied the longest chain problem for $X_1,\dots,X_n$ \iid~on $\R^d$ with density function $f:\R^d \to \R$. Under general assumptions on $f$, we showed that $n^{-\frac{1}{d}}u_n \to c_d d^{-1}U$ in $L^\infty(\R^d)$ almost surely, where $U$ is the viscosity solution of the Hamilton--Jacobi equation
\[\text{(P)} \ \ \Bigg\{\begin{aligned}
U_{x_1} \cdots U_{x_d} &{}={} f& &\text{on } \R^d_+,\\
U&{}={}0& & \text{on } \partial \R^d_+.
\end{aligned}\]
Here $\R_+ = (0,\infty)$ and $\R^d_+ = (\R_+)^d$.  

In this paper we study a fast numerical scheme for (P), first proposed in~\cite{calder2013}, and prove convergence of this scheme.  We then show how the scheme can be used to design a fast approximate non-dominated sorting algorithm, which requires access to only a fraction of the datapoints $X_1,\dots,X_n$, and we evaluate the sorting accuracy of the  new algorithm on  both synthetic and real data.  A fast approximate algorithm for non-dominated sorting has the potential to be a valuable tool for multiobjective optimization, especially in evolutionary algorithms which require frequent non-dominated sorting~\cite{deb2002}.    There are also potential applications in polynuclear growth of crystals in materials science~\cite{prahofer2000}.  Here, the scheme for (P) could be used to simulate polynuclear growth in the presence of a macroscopically inhomogeneous growth rate.  

This paper is organized as follows. In Section \ref{sec:convergence} we prove that the numerical solutions converge to the viscosity solution of (P).  We also prove a regularity result for the numerical solutions (see Lemma \ref{lem:holder}) and other important properties. In Section \ref{sec:num} we demonstrate the numerical scheme on several density functions, and in Section \ref{sec:fast} we propose a fast algorithm for approximate non-dominated sorting that is based on numerical solving (P).

\section{Numerical scheme}
\label{sec:scheme}
Let us first fix some notation.  Given $x,y \in \R^d$ we write $x\leq y$ if $x \leqq y$ and $x\neq y$.  We write $x<y$ when $x_i < y_i$ for all $i$.  For $s,t \in \R$, $\leq$ and $<$ will retain their usual definitions.  For $x \leqq y$ we define
\[[x,y] = \{z \in \R^d \, : \, x \leqq z \leqq y\}, \ \ (x,y] = \{z \in \R^d \, : \, x < z \leqq y\},\]
and make similar definitions for $[x,y)$ and $(x,y)$.
For any $x \in \R^d$ and $h>0$, there exists unique $y \in h\Z^d$ and $z \in [0,h)^d$ such that $x = y + z$.  We will denote $y$ by $\lfloor x \rfloor_h$ so that $z = x - \lfloor x \rfloor_h$. We also denote $\vb{0}=(0,\dots,0)\in \R^d$ and $\vb{1} = (1,\dots,1)\in \R^d$.  For $z \in [\vb{0},\infty)$, we denote by $\pi_z:\R^d \to [0,z]$ the projection mapping $\R^d$ onto $[\vb{0},z]$.  For $x \in [\vb{0},\infty)$ this mapping is given explicitly by
\[\pi_z(x) = (\min(x_1,z_1),\dots,\min(x_d,z_d)).\]
We say a function $u:D \subset \R^d \to \R$ is Pareto-monotone if 
\[x \leqq y  \implies u(x) \leq u(y) \ \text{for all} \ x,y \in D.\]

We now recall the numerical scheme from~\cite{calder2013}.  
Let $h>0$.   For a given $x \in [\vb{0},\infty)$, the domain of dependence for (P) is $\{y \, : \, y \leqq x\}$.  This can be seen from the connection to non-dominated sorting and the longest chain problem.   It is thus natural to consider a scheme for (P) based on backward difference quotients, yielding 
\begin{equation}\label{eq:discretize1}
\prod_{i=1}^d (U_h(x) - U_h(x - he_i)) = h^d f(x),
\end{equation}
where $U_h:h\N_0^d \to \R$ is the numerical solution of (P) and $e_1,\dots,e_d$ are the standard basis vectors in $\R^d$.
Under reasonable hypotheses on $f$, described in Section \ref{sec:convergence-proof}, there exists a unique Pareto-monotone viscosity solution of (P).  As we wish to numerically approximate this Pareto-monotone solution we may assume that $U_h(x) \geq U_h(x - he_i)$ for all $i$.  Given that $f$ is non-negative, for any $f(x),U_h(x-e_1),\dots,U_h(x-e_d)$, there is a unique $U_h(x)$ with
\[U_h(x) \geq \max(U_h(x-he_1),\dots,U_h(x-he_d)),\]
 satisfying \eqref{eq:discretize1}. Hence the numerical solution $U_h$ can be computed by visiting each grid point exactly once via any sweeping pattern that respects the partial order $\leqq$.  The scheme therefore has linear complexity in the number of gridpoints.
At each grid point, the scheme \eqref{eq:discretize1} can be solved numerically by either a binary search and/or Newton's method restricted to the interval 
\[[\max(U_h(x-he_1),\dots,U_h(x-he_d)), \max(U_h(x-he_1),\dots,U_h(x-he_d))+hf(x)^{1/d}].\]
In the case of $d=2$, we can solve the scheme \eqref{eq:discretize1} explicitly via the quadratic formula
\[U_h(x) = \frac{1}{2}(U_h(x-he_1) + U_h(x-he_2)) + \frac{1}{2}\sqrt{(U_h(x-he_1)-U_h(x-he_2))^2 + 4h^2f(x)}.\]

Now extend $U_h$ to a function $U_h:[\vb{0},\infty) \to \R$ by setting $U_h(x)=U_h(\lfloor x \rfloor_h)$.
Defining
\[\Gamma_h = [\vb{0},\infty) \setminus (h\vb{1},\infty),\]
we see that $U_h$ is a Pareto-monotone solution of the discrete scheme
\[\text{(S)} \ \ \Bigg\{\begin{aligned}
 S(h,x,U_h) &{}={} f(\lfloor x\rfloor_h),&  &\text{if } x \in (h\vb{1},\infty) \\
 U_h(x) &{}={} 0,& &\text{if }x \in \Gamma_h,
\end{aligned}\]
where $S:\R_+ \times (h\vb{1},\infty) \times X \to \R$ is defined by
\begin{equation}\label{eq:scheme-def}
S(h,x,u) =  \prod_{i=1}^d\frac{u(x) - u(x-h e_i)}{h}.
\end{equation}
Here, $X$ is the space of functions $u:[\vb{0},\infty) \to \R$.  In the next section we will study properties of solutions $U_h$ of (S).

\section{Convergence of numerical scheme}
\label{sec:convergence}

In this section we prove that the numerical solutions $U_h$ defined by (S) converge uniformly to the viscosity solution of (P).  As in~\cite{calder2013}, we place the following assumption on $f:\R^d \to [0,\infty)$:

\vspace{2mm}
\begin{itemize}
\item[(H)] There exists an open and bounded set $\Omega \subset (0,1)^d$ with Lipschitz boundary such that $f\vert_\Omega$ is Lipschitz and $\supp(f) \subset \bar{\Omega}$. 
\end{itemize}
\vspace{2mm}

\noindent It is worthwhile to take a moment to motivate the hypothesis (H).    Consider the following multi-objective optimization problem
\begin{equation}\label{eq:opt-prob}
\min \{F(x) \, : \, x \in \K\},
\end{equation}
where $F(x)=(f_1(x),\dots,f_d(x))$ with $f_i:\K \to [0,\infty)$ for all $i$, and $\K$ is the set of feasible solutions. This formulation includes many types of constrained optimization problems, where the constraints are implicitly encoded into $\K$.   If $x_1,\dots,x_n$ are feasible solutions in $\K$, then these solutions are ranked, with respect to the optimization problem \eqref{eq:opt-prob}, by performing non-dominated sorting on  $X_1=F(x_1),\dots,X_n=F(x_n)$.  Thus the domain $\Omega$ of $X_1,\dots,X_n$ is given by $\Omega = F(\K)$.  Supposing that $x_1,\dots,x_n$ are, say, uniformly distributed on $\K$, then the induced density $f$ of $X_1,\dots,X_n$ on $\R^d$ will be nonzero on $\Omega$ and identically zero on $\R^d \setminus \Omega$. Thus, the constraint that feasible solutions must lie in $\K$ directly induces a discontinuity in $f$ along $\partial \Omega$. 

In~\cite{calder2013} we showed that, under hypothesis (H), there exists a unique Pareto-monotone viscosity solution $U$ of (P) satisfying the additional boundary condition
\begin{equation}\label{eq:add-bound}
U(x)  = U(\pi_\vb{1}(x)) \ \ \text{for all } x \in [\vb{0},\infty).
\end{equation}
The boundary condition \eqref{eq:add-bound} is natural for this problem.  Indeed, since $\supp(f) \subset (0,1)^d$, there are almost surely no random variables drawn outside of $(0,1)^d$.  Hence, for any $x \in [\vb{0},\infty)$ we can write
\[u_n(x) = \max_{y \in [0,1]^d \, : \,  y\leqq x} u_n(y).\]
Since $u_n$ is Pareto-monotone, the maximum above is attained at $y=\pi_\vb{1}(x)$, and hence $u_n(x) = u_n(\pi_\vb{1}(x))$.   

For completeness, let us now give a brief outline of the proof of uniqueness for (P).  For more details, we refer the reader to~\cite{calder2013}.  The proof is based on the auxiliary  function technique, now standard in the theory of viscosity solutions~\cite{crandall1992}.  However, the technique must be modified to account for the fact that $f$ is possibly discontinuous on $\partial \Omega$, and hence does not possess the required uniform continuity.  A commonly employed technique is to modify the auxiliary function so that only a type of one-sided uniform continuity is required of $f$~\cite{tourin1992,deckelnick2004}.  This allows $f$ to, for example, have a discontinuity along a Lipschitz curve, provided the jump in $f$ is locally in the same direction (see \cite{deckelnick2004} for more details).  We cannot directly use these results because they require coercivity or uniform continuity of the Hamiltonian and/or Lipschitzness of solutions---none of which hold for (P).    Our technique for proving uniqueness for (P) employs instead an important property of viscosity solutions of (P)---namely that for any $z\in \R^d_+$, $U^z := U\circ \pi_z$ is a viscosity subsolution of (P). This property, called \emph{truncatability} in~\cite{calder2013}, follows immediately from the variational principle~\cite{calder2013}
\[U(x) = \sup_{\gamma'\geqq 0\, : \, \gamma(1)=x} \int_0^1 f(\gamma(t))^\frac{1}{d} (\gamma_1'(t) \cdots \gamma_d'(t))^\frac{1}{d} \, dt.\]
This allows us to prove a comparison principle with no additional assumptions on the Hamiltonian.

A general framework for proving convergence of a finite-difference scheme to the viscosity solution of a non-linear second order PDE was developed by Barles and Souganidis~\cite{barles1991}.  Their framework requires that the scheme be stable, monotone, consistent, and that the PDE satisfy a \emph{strong uniqueness property}~\cite{barles1991}.  The monotonicity condition is equivalent to ellipticity for second order equations, and plays a similar role for first order equations, enabling one to prove maximum and/or comparison principles for the discrete scheme.  The strong uniqueness property refers to a comparison principle that holds for semicontinuous viscosity sub- and supersolutions.

The numerical scheme (S) is easily seen to be consistent; this simply means that
\[\lim_{\substack{y\to x \\ h \to 0}} S(h,y,\phi) = \phi_{x_1}(x) \cdots \phi_{x_d}(x),\]
for all $\phi \in C^1(\R^d_+)$.  The scheme is stable~\cite{barles1991} if the numerical solutions $U_h$ are uniformly bounded in $L^\infty$, independent of $h$.   It is not immediately obvious that (S) is stable; stability follows from the discrete comparison principle for (S) (Lemma \ref{lem:discrete-comp}) and is proved in Lemma \ref{lem:holder}.  The monotonicity property requires the following:
\[S(h,x,u) \leq S(h,x,v) \ \ \text{whenever } u \geq v \text{ and } u(x)=v(x).\]
It is straightforward to verify that (S) is monotone when restricted to Pareto-monotone $u,v$.   This is sufficient since we are only interested in the Pareto-monotone viscosity solution of (P).  All that is left is to establish a strong uniqueness result for (P).  Unfortunately such a result is not available under the hypothesis (H).  Since $f$ may be discontinuous along $\partial \Omega$, we can only establish a comparison principle for continuous viscosity sub- and supersolutions (see \cite[Theorem 4]{calder2013}).  

One way to rectify this situation is to break the proof into two steps.   First prove convergence of the numerical scheme for $f$ Lipschitz on $\R^d_+$. It is straightforward in this case to establish a strong uniqueness result for (P).  Second, extend the result to $f$ satisfying (H) by an approximation argument using inf and sup convolutions.  Although this approach is fruitful, we take an alternative approach as it yields an interesting regularity property for the numerical solutions.  In particular, in Lemma \ref{lem:holder} we establish  approximate H\"older regularity of $U_h$ of the form
\begin{equation}\label{eq:holder-est}
|U_h(x) - U_h(y)| \leq C(|x-y|^\frac{1}{d} + h^\frac{1}{d}).
\end{equation}
As we verify in Appendix A, the approximate H\"older estimate \eqref{eq:holder-est} along with the stability of (S) allows us to apply the Arzel\`a-Ascoli Theorem, with a slightly modified proof, to the sequence $U_h$.  This allows us to substitute the ordinary uniqueness result from~\cite{calder2013} in place of strong uniqueness. 

\subsection{Analysis of the numerical scheme}
\label{sec:analysis}

We first prove a discrete comparison principle for the scheme (S).  This comparison principle is essential in proving stability of (S) and the approximate H\"older regularity result in Lemma \ref{lem:holder}. For the remainder of this  section, we fix $h>0$.
\begin{lemma}[Comparison principle]\label{lem:discrete-comp}
Let $z \in (h\vb{1},\infty)$ and suppose $u,v \in L^\infty_{loc}([\vb{0},\infty))$ are Pareto-monotone and satisfy
\begin{equation}\label{eq:super-sub}
S(h,x,u) \leq S(h,x,v) \ \ \text{for all } x \in (h \vb{1},z].
\end{equation}
Then $u\leq v$ on $\Gamma_h \cap [\vb{0},z]$ implies that $u\leq v$ on $[\vb{0},z]$.  
\end{lemma}
\begin{proof}
Suppose that $\sup_{[\vb{0},z]} (u-v) > 0$ and set 
\[T_r = \left\{x \in [\vb{0},\infty) \, : \, \frac{1}{d}(x_1+\cdots+x_d) \leq r \right\},\]
and
\[R  = \sup\{r > 0 \, : \, u \leq v \ \text{on} \ T_r\cap [\vb{0},z]\}.\]
Since $u\leq v$ on $\Gamma_h\cap [\vb{0},z)$ and $\sup_{[\vb{0},z]} (u-v) > 0$, we must have $ R \in [h,s] $, where $s=d^{-1}(z_1 + \cdots + z_d)$.  By the definition of $R$, there exists $x \in (h\vb{1},z]$ and $s < R$ such that
\[u(x) > v(x) \ \text{and} \ x-h e_i \in T_s \ \text{for} \ i=1,\dots,d.\]
Since $s < R$, we have $u\leq v$ on $T_s\cap[\vb{0},z]$ and hence
\begin{equation}\label{eq:comp-key}
u(x - h e_i) \leq v(x - h e_i) \leq v(x) \ \ \text{for} \ i=1,\dots,d.
\end{equation}
The second inequality above follows from Pareto-monotonicity of $v$.
Since $u$ and $v$ are Pareto-monotone and $u(x) > v(x)$ we have
\[\prod_{i=1}^d (u(x) - u(x-h e_i)) {}>{} \prod_{i=1}^d (v(x) - u(x-h e_i)) \stackrel{\eqref{eq:comp-key}}{\geq} \prod_{i=1}^d (v(x) - v(x-h e_i)).\]
Hence $S(h,x,u) > S(h,x,v)$, contradicting the hypothesis.
\end{proof}

Using the comparison principle, we can establish that numerical solutions of (S) satisfy the boundary condition at infinity \eqref{eq:add-bound}.
\begin{proposition}\label{prop:boundary_condition}
Let $u \in L_{loc}^\infty([\vb{0},\infty))$ be Pareto-monotone with $u=0$ on $\Gamma_h$. Suppose that for some $z \in (h\vb{1},\infty)$ we have
\begin{equation}\label{eq:support}
\supp\{x \mapsto S(h,x,u) \} \subset [\vb{0},z].
\end{equation}
Then we have $u = u \circ \pi_z$.
\end{proposition}
\begin{proof}
Define $v = u \circ \pi_z$ and fix $x \in [\vb{0},\infty)$.  Since $u$ is Pareto-monotone and $\pi_z(x) \leqq x$, we have $v(x) = u(\pi_z(x)) \leq u(x)$.  Hence $v \leq u$.  Since $u = v$ on $[\vb{0},z]$ we have
\[S(h,x,u) = S(h,x,v) \ \ \text{ for all } x \in [\vb{0},z]\setminus \Gamma_h.\]
For $x \not\in[\vb{0},z]\cup \Gamma_h$ we have $S(h,x,u)  = 0$ by assumption.  Since $v$ is Pareto-monotone we have $S(h,x,v) \geq 0 = S(h,x,u)$ for such $x$, and hence $S(h,x,v) \geq S(h,x,u)$ for all $x \in [\vb{0},\infty) \setminus \Gamma_h$.
Since $v=u = 0$ on $\Gamma_h$ we can apply Lemma \ref{lem:discrete-comp} to find that $u \leq v$ on $[\vb{0},\infty)$, and hence $u=v=u \circ \pi_z$.
\end{proof}

An important consequence of the comparison principle is the following approximate H\"older regularity result.  
\begin{lemma}\label{lem:holder}
Let $u \in L^\infty_{loc}([\vb{0},\infty))$ be Pareto-monotone with $u=0$ on $\Gamma_h$.  Then for any $R>0$ we have 
\begin{equation}\label{eq:cont}
|u(x) - u(y)| \leq 2d^2R^\frac{d-1}{d}\|S(h,\cdot,u)\|_{L^\infty((h,R]^d)}^\frac{1}{d}( |x - y|^\frac{1}{d} + h^\frac{1}{d})
\end{equation}
for all $x,y \in (h,R]^d$.
\end{lemma}
\begin{proof}
Let $R>0$ and $x_0,y_0 \in (h,R]^d$.  We first deal with the case where $x_0 \leqq y_0$. Set $\hat{u}(x) = u(\pi_{x_0}(x))$ and define $\psi:\R^d\to \R$ by
\begin{equation}\label{eq:psi-proofs}
\psi(x) = \begin{cases}
d(x_1\cdots x_d)^\frac{1}{d}& \text{if } x \in (\vb{0},\infty),\\
0& \text{otherwise.} \end{cases}
\end{equation}
By the concavity of $t \mapsto t^\frac{1}{d}$ we have
\[\psi(x) - \psi(x-h e_i) = d(x_1\cdots x_d)^\frac{1}{d} x_i^{-\frac{1}{d}} (x_i^\frac{1}{d} - (x_i-h)^\frac{1}{d}) {}\geq{} x_i^{-1} (x_1\cdots x_d)^\frac{1}{d} h,\]
for any $x \in (h\vb{1},\infty)$ and hence 
\begin{equation}\label{eq:discrete-super}
S(h,x,\psi) \geq 1 \ \text{ for all } \ x \in (h\vb{1},\infty).
\end{equation}
By the translation invariance of $S$ and \eqref{eq:discrete-super} we have 
\begin{equation}\label{eq:h-super}
S(h,x,\psi(\cdot - b)) \geq 1 \ \text{ for all } b \in [\vb{0},\infty), \  x \in (b + h\vb{1},\infty).
\end{equation}
Set $b^i = (x_{0,i} - h)e_i \in \R^d$.  For $x \in [\vb{0},\infty)$ set
\[w(x) = \hat{u}(x) + \|S(h,\cdot,u)\|^\frac{1}{d}_{L^\infty((h,R]^d)}\sum_{i=1}^d \psi(x-b^i),\]
and note that $w$ is Pareto-monotone.  Let $x \in (h\vb{1},\infty)\setminus(h\vb{1},x_0]$.  Then for some $k$ we have $x_k > x_{0,k}$, and hence $x > b^k + h\vb{1}$.  We therefore have
\begin{align*}
S(h,x,w) &{}\geq{} \frac{1}{h^d} \prod_{i=1}^d \Bigg( \hat{u}(x) - \hat{u}(x-h e_i) \\
& \hspace{2cm}+ \|S(h,\cdot,u)\|^\frac{1}{d}_{L^\infty((h,R]^d)} (\psi(x-b^k) - \psi(x-b^k - h e_i))\Bigg) \\
&{}\geq{}S(h,x,\hat{u}) +  \|S(h,\cdot,u)\|_{L^\infty((h,R]^d)} S(h,x,\psi(\cdot - b^k))\\
&{}\hspace{-2.4mm}\stackrel{\eqref{eq:h-super}}{\geq}{} \hspace{-2.4mm}S(h,x,\hat{u}) +  \|S(h,\cdot,u)\|_{L^\infty((h,R]^d)} \\
&{}\geq{} S(h,x,u).
\end{align*}
Suppose now that $x \in (h\vb{1},x_0]$.  Then since $u=\hat{u}$ on $[\vb{0},x_0]$ we have $S(h,x,\hat{u}) = S(h,x,u)$ and hence $S(h,x,w) \geq S(h,x,u)$.  Since $w \geq u=0$ on $\Gamma_h\cap[0,R]^d$, we can apply Lemma \ref{lem:discrete-comp} to obtain $w \geq u$ on $[0,R]^d$, which yields
\begin{align}\label{eq:case1}
u(y_0) - \hat{u}(y_0) &{}\leq{} \|S(h,\cdot,u)\|_{L^\infty((h,R]^d)}^\frac{1}{d} \sum_{i=1}^d \psi(y_0 - b^i) \notag \\
&{}\leq{} dR^\frac{d-1}{d}\|S(h,\cdot,u)\|_{L^\infty((h,R]^d)}^\frac{1}{d} \sum_{i=1}^d (y_{0,i} - x_{0,i} + h)^\frac{1}{d}\notag \\
&{}\leq{} d^2R^\frac{d-1}{d}\|S(h,\cdot,u)\|_{L^\infty((h,R]^d)}^\frac{1}{d} (|x_0 - y_0|^\frac{1}{d} +  h^\frac{1}{d}).
\end{align}
Noting that $\pi_{x_0}(y_0) = x_0$ we have $\hat{u}(y_0) = u(\pi_{x_0}(y_0)) = u(x_0)$, which completes the proof for the case that $x_0 \leqq y_0$.  

Suppose now that $x_0,y_0 \in (h,R]^d$ such that $x_0 \not\leqq y_0$.  Set 
\[x = \pi_{x_0}(y_0) = \pi_{y_0}(x_0).\]
Then $|x_0-x| \leq |x_0 - y_0|$, $|y_0 - x| \leq |x_0 - y_0|$, $x \leqq x_0$, and $x \leqq y_0$.  It follows that
\begin{align*}
|u(x_0)-u(y_0)| &{}\leq{} |u(x_0) - u(x)| + |u(y_0) - u(x)| \\
&{}\leq{} 2d^2R^\frac{d-1}{d}\|S(h,\cdot,u)\|_{L^\infty((h,R]^d)}^\frac{1}{d}( |x_0 - y_0|^\frac{1}{d} + h^\frac{1}{d}),
\end{align*}
which completes the proof.
\end{proof}

\subsection{Convergence of numerical scheme}
\label{sec:convergence-proof}

Our main result is the following convergence statement for the scheme (S).
\begin{theorem}\label{thm:conv}
Let $f$ be nonnegative and satisfy (H).  Let $U$ be the unique Pareto-monotone viscosity solution of (P) satisfying \eqref{eq:add-bound}.  For every $h>0$ let $U_h:[\vb{0},\infty) \to \R$ be the unique Pareto-monotone solution of (S).  Then $U_h \to U$ uniformly on $[\vb{0},\infty)$ as $h\to 0$.
\end{theorem}
\begin{proof}
By (H) we have that $f(x) = 0$ for $x\not\in(0,1)^d$, and hence $\supp(f(\lfloor \cdot \rfloor_h)) \subset [0,1]^d$.  Therefore, by Proposition \ref{prop:boundary_condition}, we have that $U_h$ satisfies \eqref{eq:add-bound}. Combining this with Lemma \ref{lem:holder}  we have
\begin{equation}\label{eq:stability}
\|U_h\|_{L^\infty([\vb{0},\infty))} \leq C\|f\|_{L^\infty([\vb{0},\infty))}^\frac{1}{d},
\end{equation}
for all $h>0$.  Similarly, combining \eqref{eq:add-bound} with Lemma \ref{lem:holder} we have
\begin{equation}\label{eq:holder}
|U_h(x) - U_h(y)|\leq 2d^2\|f\|_{L^\infty([\vb{0},\infty))}^\frac{1}{d}(|x-y|^\frac{1}{d} + h^\frac{1}{d}) \ \ \text{for all } x,y \in [\vb{0},\infty),
\end{equation}
 for every $h>0$.
The estimates in \eqref{eq:stability} and \eqref{eq:holder} show uniform boundedness, and a type of equicontinuity, respectively, for the sequence $U_h$.  By an argument similar to the proof of the Arzel\`a-Ascoli Theorem (see the Appendix), there exists a subsequence $h_k \to 0$ and $u \in C^{0,\frac{1}{d}}([\vb{0},\infty))$ such that $U_{h_k} \to u$ uniformly on compact sets in $[\vb{0},\infty)$.  By \eqref{eq:add-bound}, we actually have $U_{h_k} \to u$ uniformly on $[\vb{0},\infty)$.   Since the scheme (S) is monotone and consistent, it is a standard result that $u$ is a viscosity solution of (P)~\cite{barles1991}.  Note that $U_h$ is Pareto-monotone, $U_h=0$ on $\Gamma_h$, and $U_h$ satisfies \eqref{eq:add-bound}.  Since $U_{h_k} \to u$ uniformly, it follows that $u$ is Pareto-monotone, $u=0$ on $\partial \R^d_+$, and $u$ satisfies \eqref{eq:add-bound}. 
By uniqueness for (P)~\cite[Theorem 5]{calder2013} we have $u=U$.    Since we can apply the same argument to any subsequence of $U_h$, it follows that $U_h \to U$ uniformly on $[\vb{0},\infty)$.
\end{proof}

In Section \ref{sec:num}, we observe that the numerical scheme provides a fairly consistent underestimate of the exact solution of (P).  The following lemma shows that this is indeed the case whenever the solution $U$ of (P) is concave.
\begin{lemma}\label{lem:conv-below}
Let $f$ be nonnegative and satisfy (H).  Let $U$ be the unique Pareto-monotone viscosity solution of (P) satisfying \eqref{eq:add-bound}.  For every $h>0$ let $U_h:[\vb{0},\infty) \to \R$ be the unique Pareto-monotone solution of (S).  If $U$ is concave on $[\vb{0},\infty)$ then $U_h \leq U$ for every $h>0$.
\end{lemma}
\begin{proof}
Fix $h>0$.  Since $U$ is concave, it is differentiable almost everywhere.\footnote{The fact that $U$ is Pareto-monotone also implies differentiability almost everywhere.}  Let $x \in (h\vb{1},\infty)$ be a point at which $U$ is differentiable and $f$ is continuous. Since $U$ is concave we have
\[U(x) - U(x-he_i) \geq hU_{x_i}(x) \ \ \text{for all } i.\]
Since $U$ is a viscosity solution of (P) and $f$ is continuous at $x$ we have
\[S(h,x,U) \geq U_{x_1}(x) \cdots U_{x_d}(x) =  f(x). \]
Since $x\mapsto S(h,x,U)$ is continuous, we see that $S(h,x,U) \geq f_*(x)=f(x)$ for all $x \in (h\vb{1},\infty]$.  Now define $W_h(x) = U(\lfloor x\rfloor_h)$.  Then we have
\[S(h,x,W_h) \geq f(\lfloor x\rfloor_h) \ \ \text{for all } x \in (h\vb{1},x],\]
and $W_h = 0$ on $\Gamma_h$. It follows from Lemma \ref{lem:discrete-comp} that $U_h \leq W_h$.  Since $U$ is Pareto-monotone, we have $W_h \leq U$, which completes the proof.
\end{proof}

\section{Numerical Results}
\label{sec:num}

We now present some numerical results using the scheme (S) to approximate the viscosity solution of (P).  We consider four special cases where the exact solution of (P) can be expressed in analytical form. Let $f_1(x) = 1$, $f_2(x) = \frac{2^d}{\pi^{d/2}} e^{-|x|^2}$,  
\begin{align*}
f_3(x) &{}={} 1-\chi_{[0,1/2]^d}(x) \ \ \  \text{and} \ \ \ f_4(x) {}={} \left(\sum_{i=1}^d x_i^9\right)^{1-d} \prod_{i=1}^d \left( 9 x_i^9 + \sum_{i=1}^d x_i^9\right).
\end{align*}
Here, $\chi_A$ denotes the characteristic function of the set $A$.
The corresponding solutions of (P) are $U_1(x) = d(x_1\cdots x_d)^{\frac{1}{d}}$, $U_2(x) = d\left(\prod_{i=1}^d {\rm erf} \,(x_i)\right)^\frac{1}{d}$,
and
\begin{align*}
U_3(x) &{}={} d\max_{i\in \{1,\dots,d\}} \left\{ \left(x_i - \frac{1}{2}\right)_+\prod_{j \neq i} x_j \right\}^\frac{1}{d}, \ \ \ U_4(x) {}={} d\left(\prod_{i=1}^d x_i \cdot \sum_{i=1}^d x_i^9\right)^\frac{1}{d},
\end{align*}
where ${\rm erf}\,(x)$ is the error function defined by ${\rm erf}\,(x) = 2/\sqrt{\pi} \int_0^x e^{-t^2} \, dt$, and $x_+ :=\max(0,x)$.
The solutions $U_1$ and $U_2$ are special cases of the formula 
\begin{equation}\label{eq:special-case}
U(x) = d \left(\int_{[\vb{0},x]} f(y) \, dy\right)^\frac{1}{d}, 
\end{equation}
which holds when $f$ is separable, i.e., $f(x)=f_1(x_1)\cdots f_d(x_d)$~\cite{calder2013}.
The solution $U_3$ can be obtained by the method of characteristics.  We chose to evaluate the proposed numerical scheme for $U_4$ because it has non-convex level sets, and then computed $f_4$ via (P).  In the probabilistic interpretation of (P) as the continuum limit of non-dominated sorting, non-convex Pareto fronts play an important role~\cite{ehrgott2005,calder2013}.

\begin{figure}[!t]
\centering
\subfigure[$U_1$]{\includegraphics[clip = true, trim = 20 20 20 20, width=0.45\textwidth]{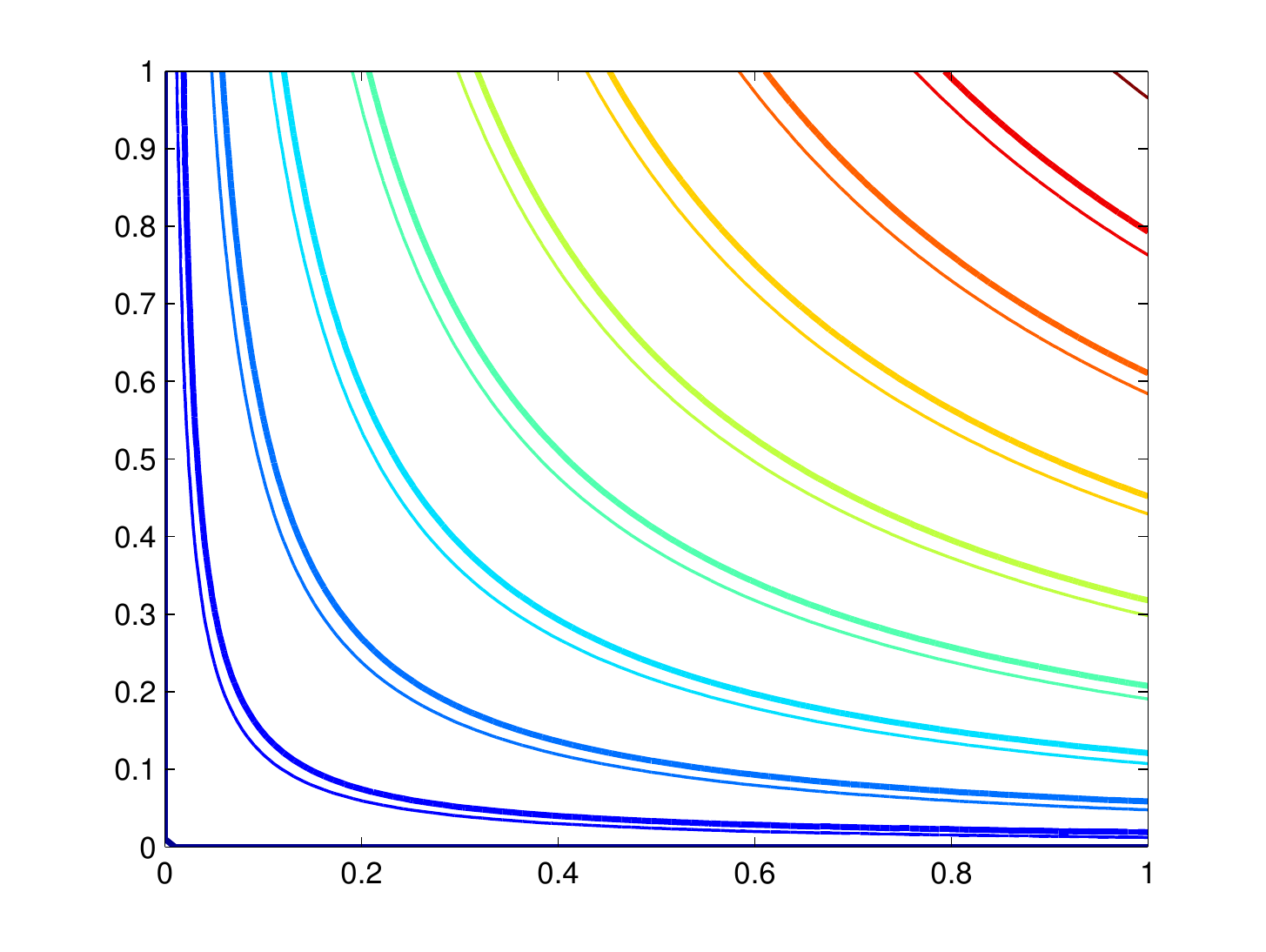}}
\subfigure[$U_2$]{\includegraphics[clip = true, trim = 20 20 20 20, width=0.45\textwidth]{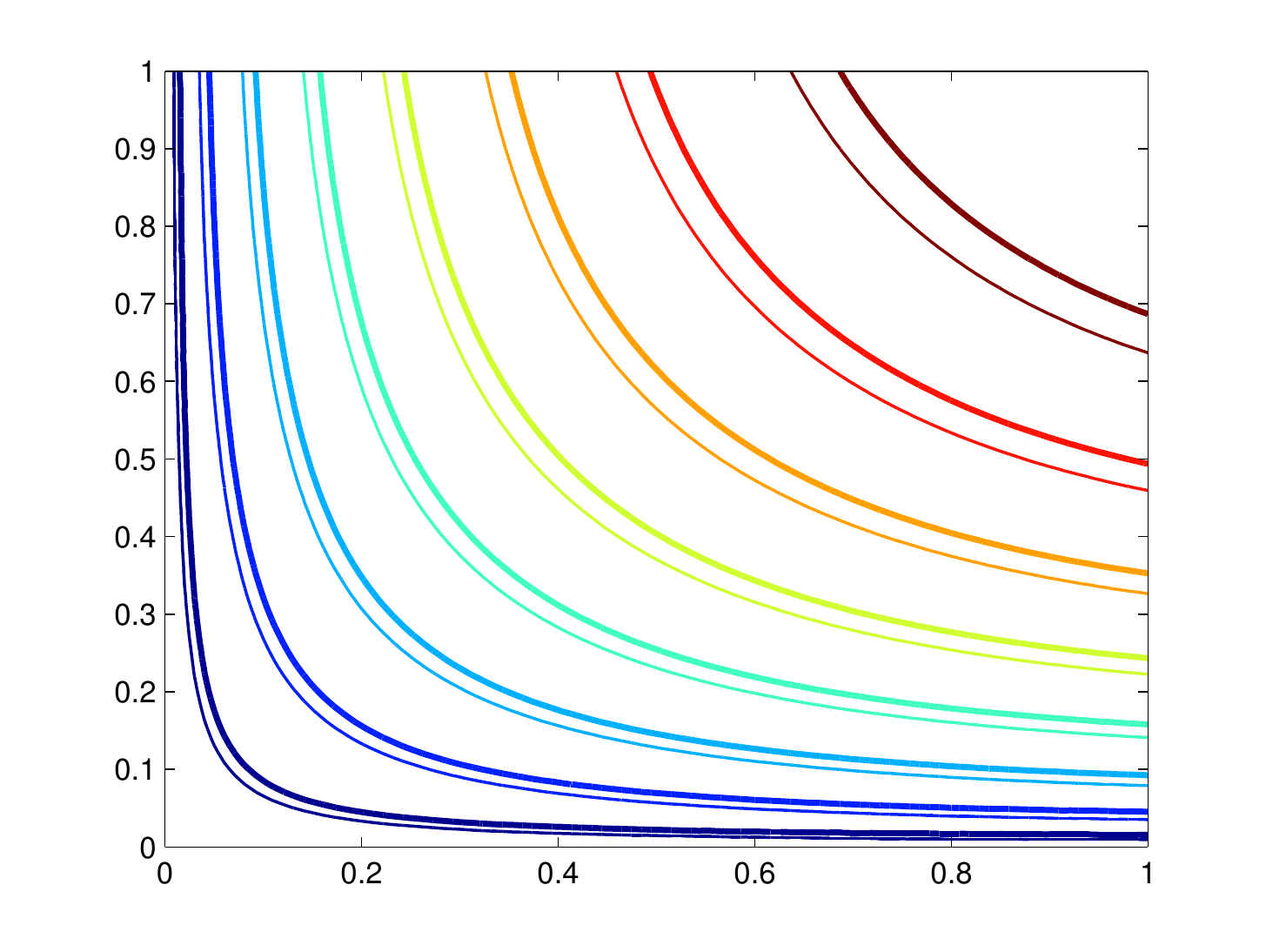}}
\subfigure[$U_3$]{\includegraphics[clip = true, trim = 20 20 20 20, width=0.45\textwidth]{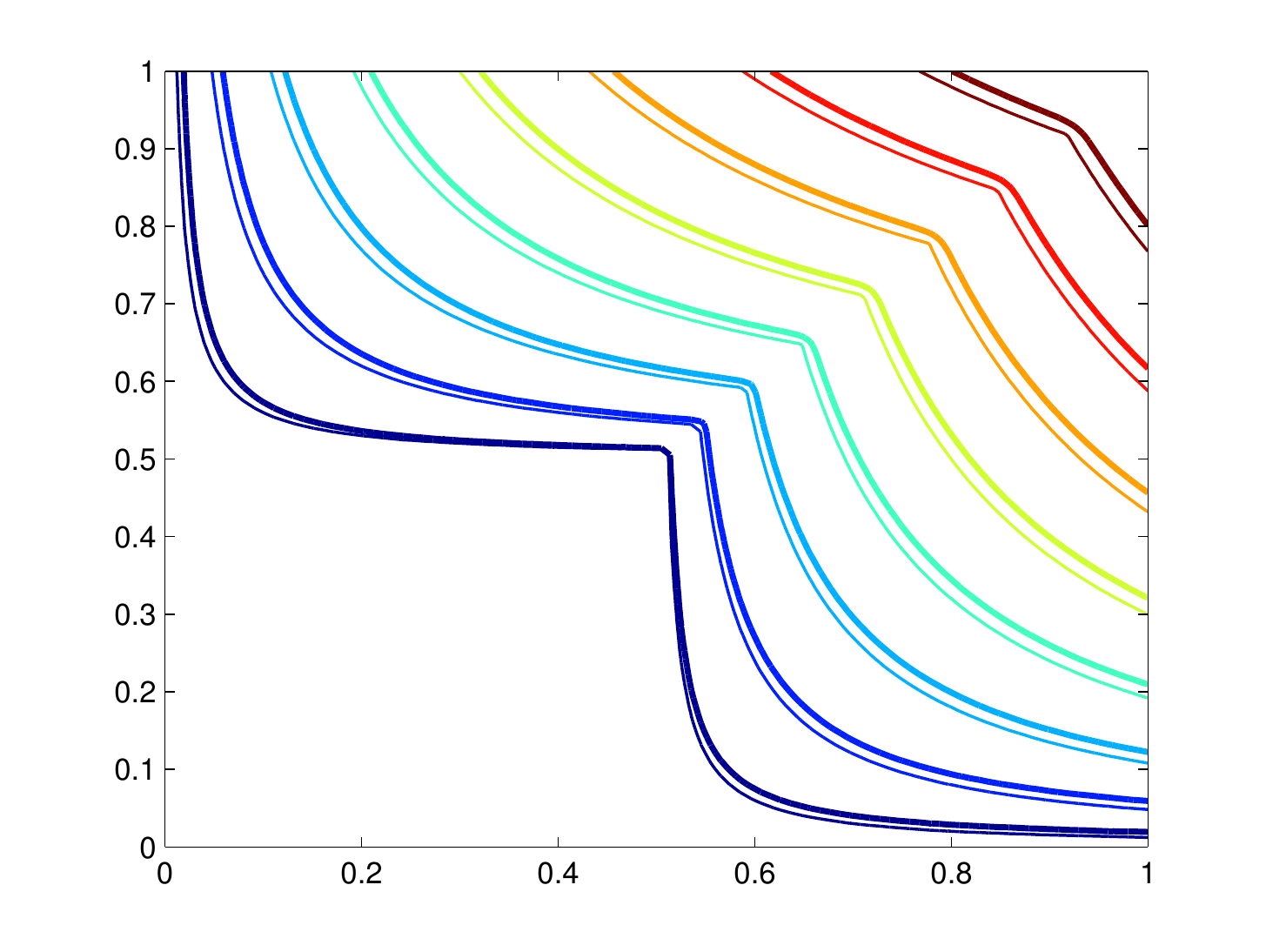}}
\subfigure[$U_4$]{\includegraphics[clip = true, trim = 20 20 20 20, width=0.45\textwidth]{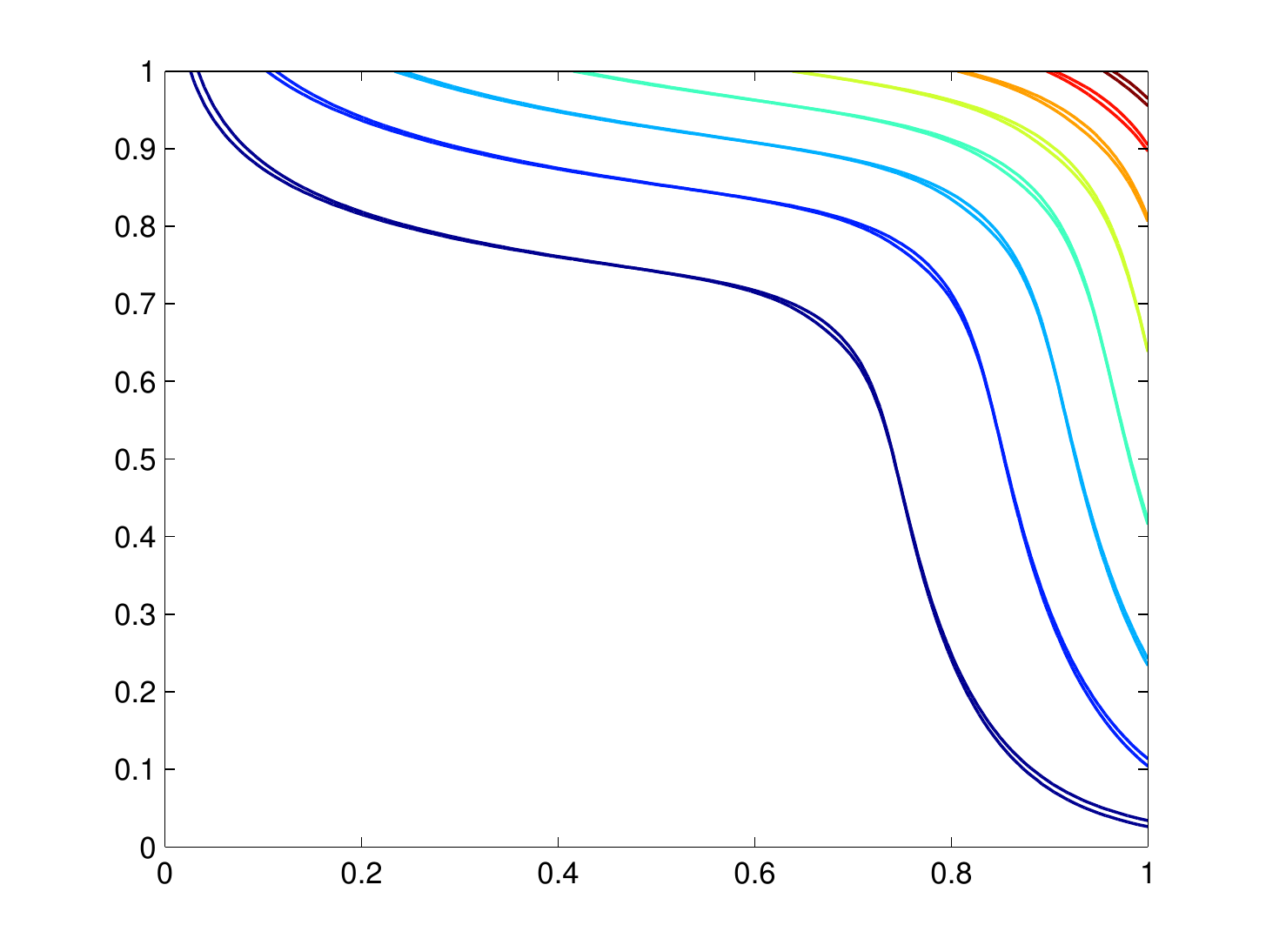}\label{fig:nonconvex-2D}}
\caption{Comparison of numerical solutions and exact solutions of (P) for $d=2$.  The thin and thick lines represent the level sets of the exact and numerical solutions, respectively.}
\label{fig:2D-sim}
\end{figure}

\begin{figure}[!t]
\centering
\subfigure[$U_1$]{\includegraphics[clip = true, trim = 30 25 30 25, width=0.45\textwidth]{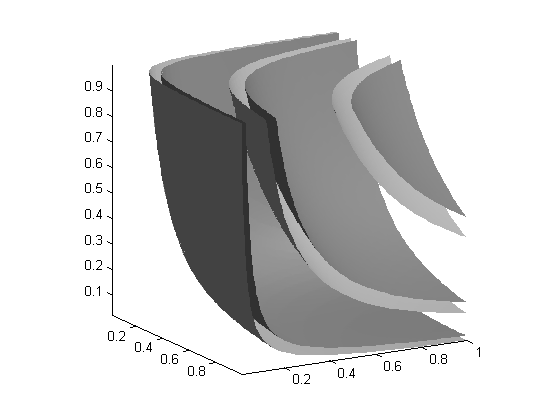}}
\subfigure[$U_2$]{\includegraphics[clip = true, trim = 30 25 30 25, width=0.45\textwidth]{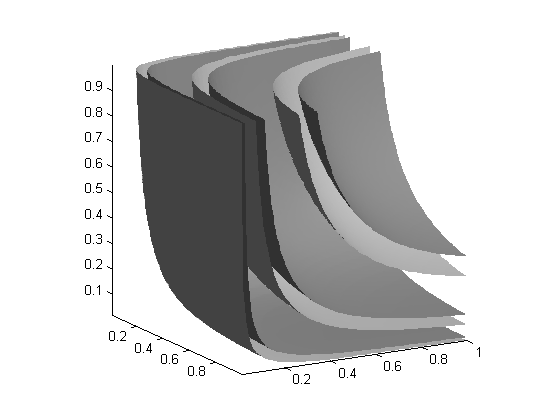}}
\subfigure[$U_3$]{\includegraphics[clip = true, trim = 30 25 30 25, width=0.45\textwidth]{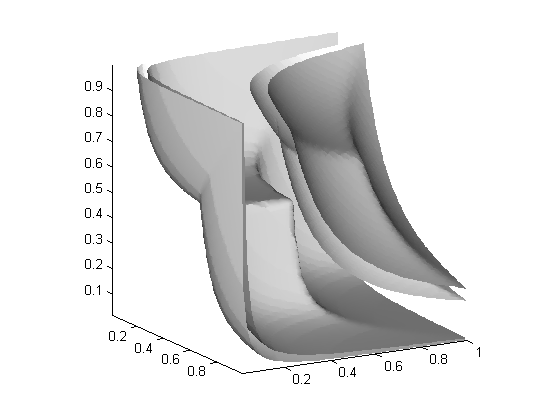}}
\subfigure[$U_4$]{\includegraphics[clip = true, trim = 30 25 30 25, width=0.45\textwidth]{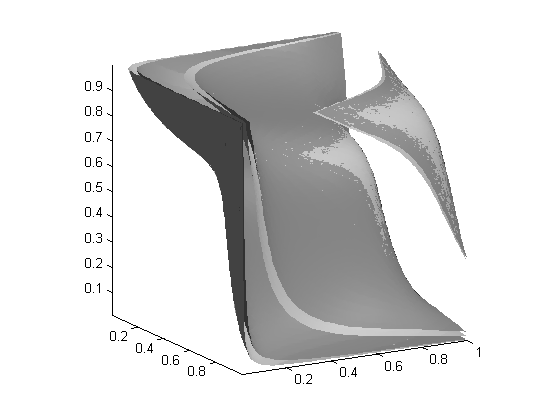}}
\caption{Comparison of numerical solutions and exact solutions of (P) for $d=3$.  The light and dark surfaces represent the level sets of the exact and numerical solutions, respectively. }
\label{fig:3D-sim}
\end{figure}

We computed the numerical solutions for $d=2$ and $d=3$.  For $d=2$ we used a $100\times 100$ grid, and for $d=3$, we used a $50\times 50\times 50$ grid and solved the scheme at each grid point via a binary search with precision $\eps=10^{-4}$.   Figures \ref{fig:2D-sim} and \ref{fig:3D-sim} compare the level sets of the exact solutions to those of the numerical solutions for $d=2$ and $d=3$, respectively.  In Figure \ref{fig:2D-sim}, the thin lines correspond to the exact solution while the thick lines correspond to the numerical solutions, with the exception of \ref{fig:nonconvex-2D} where both are thin lines for increased visibility.  In Figure \ref{fig:3D-sim}, the darker surfaces correspond to the numerical solution while the lighter surfaces represent the exact solution.  For both $d=2$ and $d=3$, we can see that the level sets of the numerical solutions consistently overestimate the true solution, indicating that the numerical solutions are converging from below to the exact solutions.  We proved in Lemma \ref{lem:conv-below} that $U_h\nnearrow U$ whenever $U$ is concave, so this observation is to be expected. Note however, that $U_3$ is not convex, yet the overestimation is still present, indicating that Lemma \ref{lem:conv-below} may hold under more general hypotheses on $U$.  We also observe that $U_3$ has a shock, which is resolved reasonably well for $d=2$ and $d=3$, given the grid sizes used.  

\subsection{Rate of convergence}
\label{sec:pde-rates}

We show here the results of some numerical experiments concerning the rate of convergence of $U_h \to U$ and $ n^{-\frac{1}{d}}u_n \to c_dd^{-1} U$.  Figure \ref{fig:pde-rates} shows $\|U_h - U\|_{L^1([0,1]^2)}$ and $\|U_h - U\|_{L^\infty(\R^d_+)}$ versus $h$ for the density $f_3(x) = 1-\chi_{[0,1/2]^d}(x)$ from the beginning of Section \ref{sec:num}.    Both norms appear to have convergence rates on the order of $O(h^\alpha)$, and a regression analysis yields $\alpha = 0.5006$ for the $L^\infty$ norm and $\alpha = 0.8787$ for the $L^1$ norm.  Thus, it is reasonable to suspect an $L^\infty$ convergence rate of the form
\begin{equation}\label{eq:pde-rates}
\|U_h - U\|_{L^\infty(\R^d_+)}  \leq C h^\frac{1}{d},
\end{equation}
for some constant $C>0$.  We intend to investigate this in a future work.  It is quite natural that the convergence rate for the $L^1$ norm is substantially better than the $L^\infty$ norm, due to the non-differentiability of $U_3$ at the boundary $\partial \R^2_+$.  This induces a large error near $\partial \R^2_+$ which has a more significant impact on the $L^\infty$ norm.  
\begin{figure}
\centering
\subfigure[Convergence rate of scheme (S)]{\includegraphics[clip=true,trim = 10 5 47 30, width=0.475\textwidth]{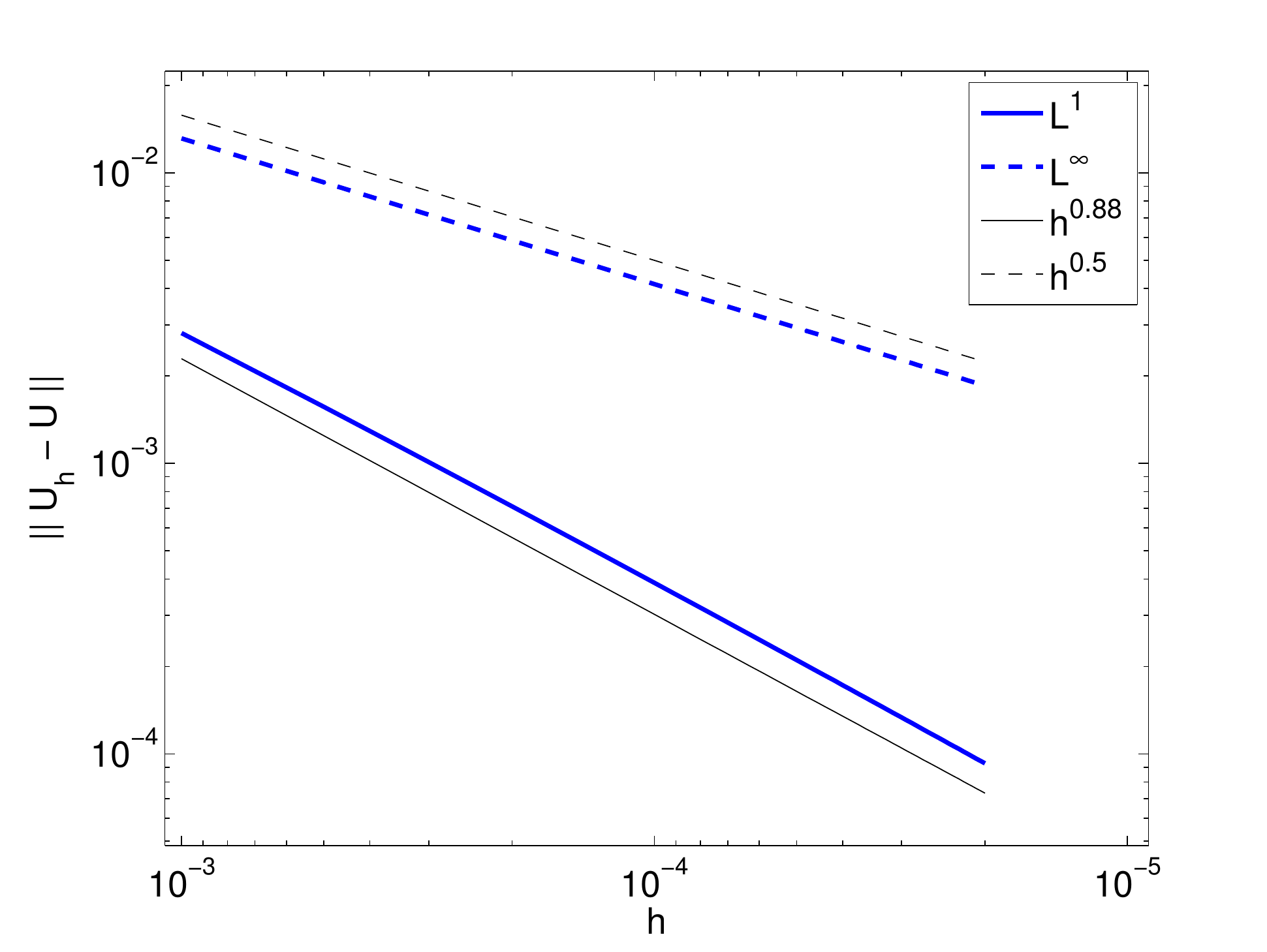}\label{fig:pde-rates}}
\subfigure[Stochastic convergence rates]{\includegraphics[clip=true,trim = 5 5 47 30, width=0.475\textwidth]{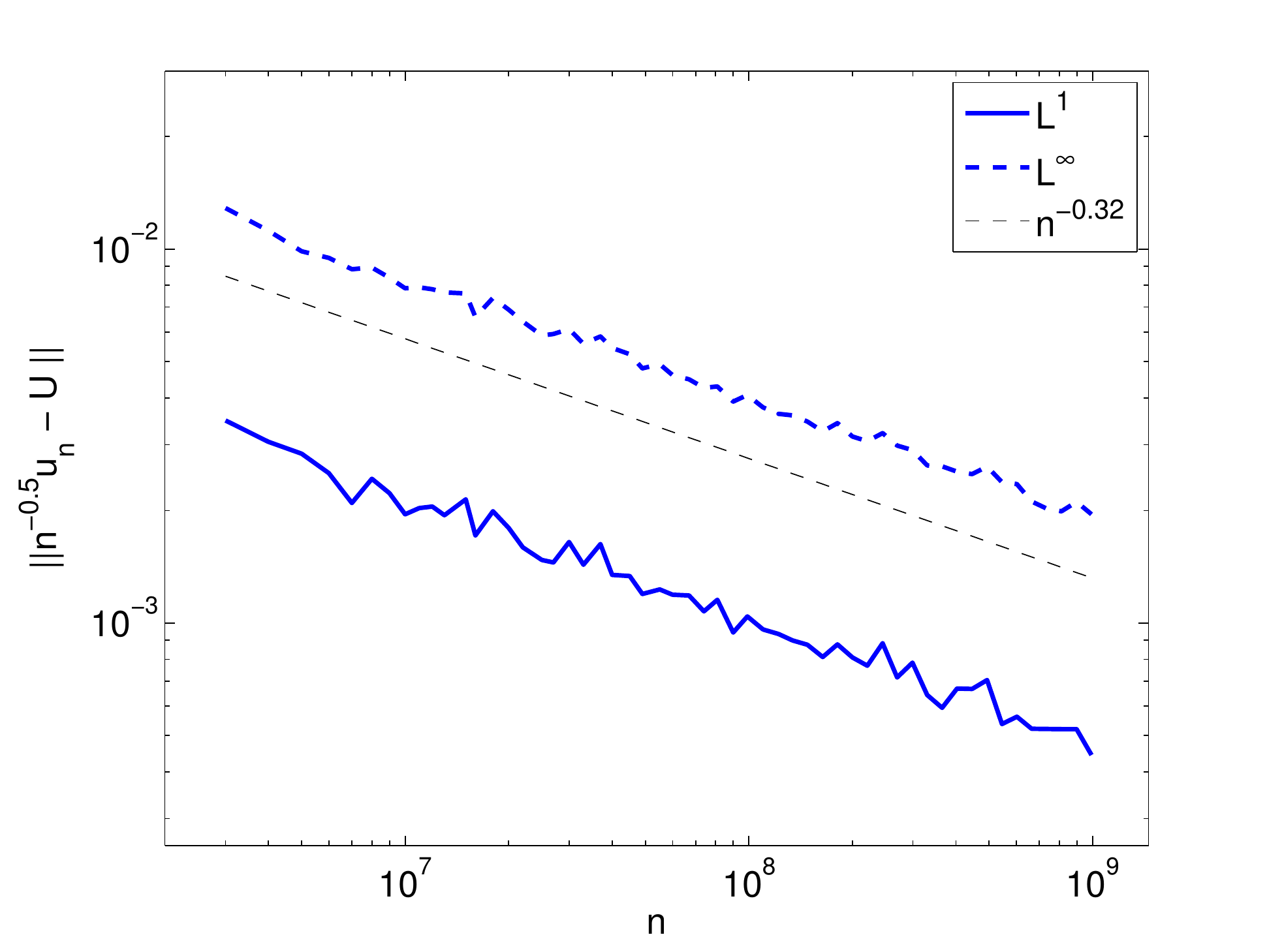}\label{fig:sto-rates}}
\caption{Convergence rates for (a) the scheme (S) as a function of the grid resolution $h$,  and (b) the stochastic convergence $n^{-\frac{1}{d}}u_n \to c_d d^{-1} U$ as a function of the number $n$ of random samples.}
\label{fig:rates}
\end{figure}
  
To measure the rate of convergence of $n^{-\frac{1}{d}}u_n \to c_d d^{-1} U$, we consider the following two norms
\begin{equation}\label{eq:max-norm}
|n^{-\frac{1}{d}}u_n - c_dd^{-1} U|_{L^\infty} := \max_{1 \leq i \leq n} |n^{-\frac{1}{d}}u_n(X_i) - c_dd^{-1} U(X_i)|
\end{equation}
and
\begin{equation}\label{eq:L1-norm}
|n^{-\frac{1}{d}}u_n - c_dd^{-1} U|_{L^1} := \frac{1}{n} \sum_{i=1}^n |n^{-\frac{1}{d}}u_n(X_i) - c_dd^{-1} U(X_i)|
\end{equation}
Figure \ref{fig:sto-rates} shows \eqref{eq:max-norm} and \eqref{eq:L1-norm} versus $n$ for the same density $f_3$.  For each $n$ the values of \eqref{eq:max-norm} and \eqref{eq:L1-norm} were computed by taking the average over $10$ independent realizations.  It appears that both norms decay on the order of $O(n^{-\alpha})$, and a regression analysis yields $\alpha =0.3281$ for the $L^1$ norm \eqref{eq:L1-norm} and $\alpha=0.3144$ for the $L^\infty$ norm \eqref{eq:max-norm}. These results are in line with the known convergence rates for the longest chain problem with a uniform distribution on $[0,1]^d$~\cite{bollobas1992}.
 
The results for the other densities $f_1,f_2,$ and $f_4$ are similar.  We demonstrated the convergence rates on $f_3$ due to the fact that it has many important features; namely, it is discontinuous, yields non-convex Pareto-fronts, and induces a shock in the viscosity solution $U_3$ of (P).

\section{Fast approximate non-dominated sorting}
\label{sec:fast}

We demonstrate now how the numerical scheme (S) can be used for fast approximate non-dominated sorting, and give a real-world application to anomaly detection in Section \ref{sec:real}.  We assume here that the given data $X_1,\dots,X_n$ are drawn \iid~from a reasonably smooth density function $f$, and that $n$ is large enough so that $n^{-\frac{1}{d}}u_n$ is well approximated by $c_dd^{-1} U$.   In this regime, it is reasonable to consider an approximate non-dominated sorting algorithm based on numerically solving (P).  A natural algorithm is as follows.   

Since the density $f$ is rarely known in practice, the first step is to form an estimate $\widehat{f}$ of $f$ using the samples $X_1,\dots,X_n$.  In the large sample regime, this can be done very accurately using, for example, a kernel density estimator~\cite{tsybakov2009} or a $k$-nearest neighbor estimator~\cite{loftsgaarden1965}.  To keep the algorithm as simple as possible, we opt for a simple histogram to estimate $f$, aligned with the same grid used for numerically solving (P).  When $n$ is large, the estimation of $f$ can be done with only a random subset of $X_1,\dots,X_n$ of cardinality $k\ll n$, which avoids considering all $n$ samples.  The second step is to use the numerical scheme (S) to solve (P) on a fixed grid of size $h$, using the estimated density $\widehat{f}$ on the right hand side of (P). This yields an estimate $\widehat{U}_h$ of $U$, and the final step is to evaluate $\widehat{U}_h$ at each sample $X_1,\dots,X_n$ to yield approximate Pareto ranks for each point.  The final evaluation step can be viewed as an interpolation; we know the values of $\widehat{U}_h$ on each grid point and wish to evaluate $\widehat{U}_h$ at an arbitrary point.  A simple linear interpolation is sufficient for this step.  However, in the spirit of utilizing the PDE (P), we solve the scheme (S) at each point $X_1,\dots,X_n$ using the values of $\widehat{U}_h$ at neighboring grid points, i.e., given $\widehat{U}_h(x-he_i)$ for all $i$, and $y \in [x-h\vb{1},x]$, we compute $\widehat{U}_h(y)$ by solving
\begin{equation}\label{eq:scheme-sub}
\prod_{i=1}^d (\widehat{U}_h(y) - \widehat{U}_h(y - h_ie_i)) = h_1\cdots h_d \widehat{f}(x),
\end{equation}
where $h_i = y_i - (x_i-h)$.  In \eqref{eq:scheme-sub} we compute $\widehat{U}_h(y-h_ie_i)$ by linear interpolation using adjacent grid points.  Figure \ref{fig:subgrid} illustrates the grid used for computing $\widehat{U}_h(y)$.  
\begin{figure}
\centering
\includegraphics[width=0.45\textwidth]{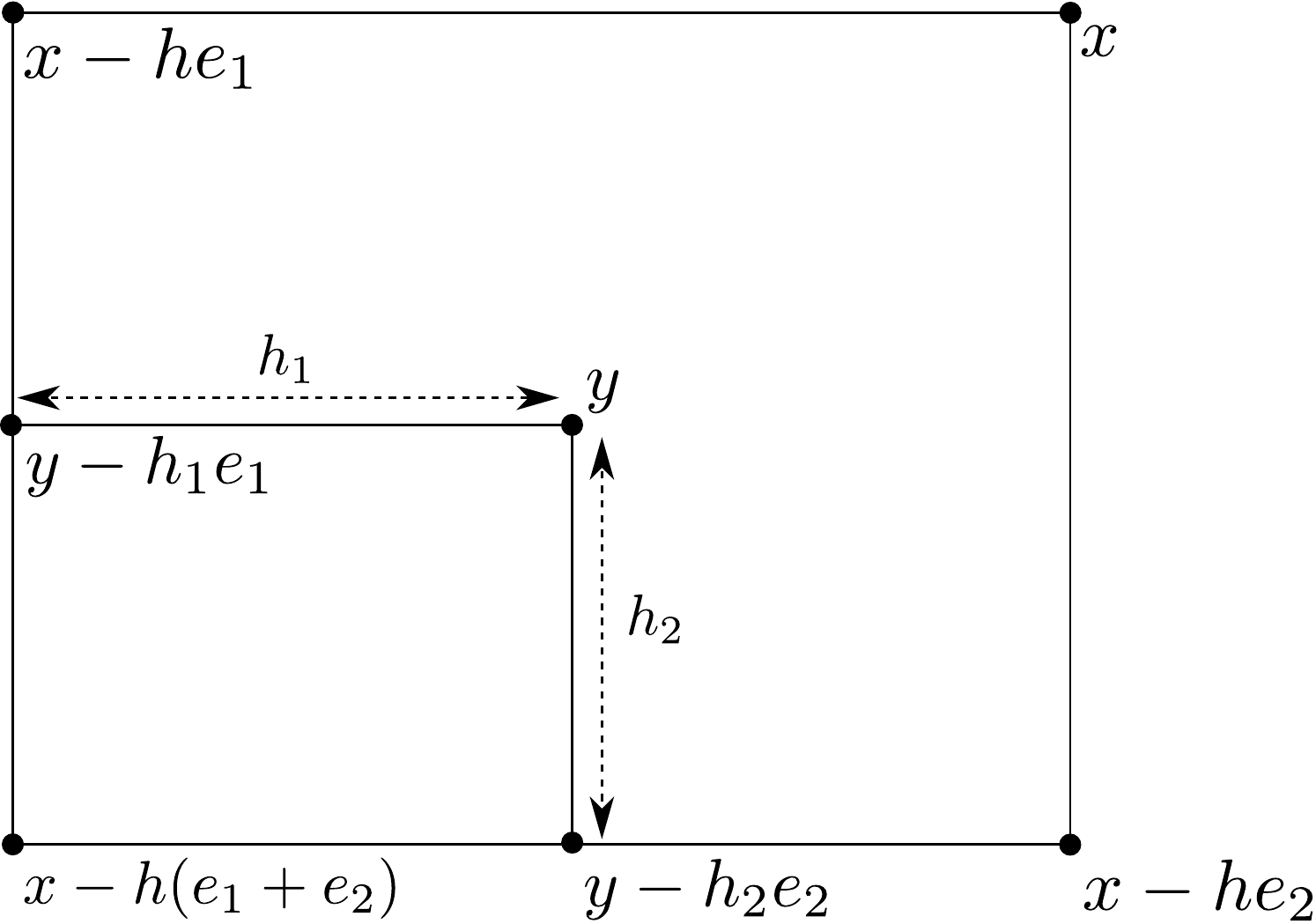}
\caption{Depiction of the grid used for computing $\widehat{U}_h(y)$ according to \eqref{eq:scheme-sub}.  The values of $\widehat{U}_h(y-h_1e_1)$ and $\widehat{U}_h(y-h_2e_2)$ are computed by linear interpolation using adjacent grid points, i.e., $\widehat{U}_h(y-h_1e_1)$ is computed via linearly interpolating between $\widehat{U}_h(x-he_1)$ and $\widehat{U}_h(x-h(e_1 + e_2))$.}
\label{fig:subgrid}
\end{figure}

The entire algorithm is summarized in Algorithm \ref{alg:ndom}. 
\begin{algorithm}\label{alg:ndom}
Fast approximate non-dominated sorting
\begin{enumerate}
\item Select $k$ points from $X_1,\dots,X_n$ at random.  Call them $Y_1,\dots,Y_k$.
\item Select a grid spacing $h$ for solving the PDE and estimate $f$ with a histogram aligned to the grid $h\N_0^d$, i.e.,
\begin{equation}\label{eq:fhat}
\widehat{f}_h(x) = \frac{1}{kh^d} \cdot \#\bigg\{Y_i \, : \, x \leqq Y_i \leqq x + h\vb{1}\bigg\} \ \text{for } x \in h\N_0^d.
\end{equation}
\item Compute the numerical solution $\widehat{U}_h$ on $h\N_0^d \cap [0,1]^d$ via (S).
\item Evaluate $\widehat{U}_h(X_i)$ for $i=1,\dots,n$ via interpolation. 
\end{enumerate}
\end{algorithm}
For simplicity of discussion, we have assumed that $X_1,\dots,X_n$ are drawn from $[0,1]^d$, but this is not essential as the scheme (S) can be easily adapted to any hypercube in $\R^d$, and this is in fact what we do in our implementation of Algorithm \ref{alg:ndom}. 

\subsection{Convergence rates in Algorithm \ref{alg:ndom}}
\label{sec:sublinear}

It is important to understand how the parameters $k$ and $h$ in Algorithm \ref{alg:ndom} affect the accuracy of the estimate $\hat{U}_h$.   We first consider the estimate $\hat{f}_h$.  By \eqref{eq:fhat}, we can write
\[h^d\hat{f}_h(x) = \frac{1}{k}\sum_{i=1}^k \chi_{[x,x+h\vb{1}]}(Y_i).\]
Hence $h^d\hat{f}_h(x)$ is the average of \iid~Bernoulli random variables with parameter
\begin{equation}\label{eq:p}
p = \int_{[x,x+h\vb{1}]} f(y) \, dy.
\end{equation}
By the central limit theorem, the fluctuations of $\hat{f}_h(x)$ about its mean satisfy
\begin{equation}\label{eq:rate1}
\left|\hat{f}_h(x) - \frac{p}{h^d}\right| \leq C\frac{1}{\sqrt{k}h^d},
\end{equation}
with high probability.

Let us suppose now that $f$ is globally Lipschitz.  The following can be easily modified for $f$ more or less regular, yielding similar results.  Then by \eqref{eq:p} we have
\[\left|f(x) - \frac{p}{h^d}\right| \leq C\sqrt{d}  h.\]
Combining this with \eqref{eq:rate1} we have
\begin{equation}\label{eq:rate3}
\|\hat{f}_h - f \|_{L^\infty([0,1]^d\cap h\N^d)} \leq C\left( \frac{1}{\sqrt{k}h^d} + h\right),
\end{equation}
with high probability.  By the discrete comparison principle (Lemma \ref{lem:discrete-comp}) and \eqref{eq:rate3} we have that
\begin{equation}\label{eq:rate4}
\|\hat{U}_h - U_h\|_{L^\infty([0,1]^d)} \leq d\|\hat{f}_h - f\|^\frac{1}{d}_{L^\infty([0,1]^d\cap h\N^d)} \leq C\left(k^{-\frac{1}{2d}}h^{-1} + h^\frac{1}{d}\right),
\end{equation}
with high probability.  Based on the numerical evidence presented in Section \ref{sec:pde-rates}, it is reasonable to suspect that $\|U - U_h\|_{L^\infty([0,1]^d)} \leq Ch^\frac{1}{d}$.  If this is indeed the case, then in light of \eqref{eq:rate4} we have
\begin{equation}\label{eq:final-rate}
\|\hat{U}_h - U\|_{L^\infty([0,1]^d)} \leq C\left(k^{-\frac{1}{2d}}h^{-1} + h^\frac{1}{d}\right),
\end{equation}
with high probability. 

The right side of the inequality \eqref{eq:final-rate} is composed of two competing additive terms. The first term $Ck^{-\frac{1}{2d}}h^{-1}$ captures the effect of random errors (variance) due to an insufficient number $k$ of samples. The second term $Ch^{\frac{1}{d}}$ captures the effect of non-random errors (bias) due to insufficient resolution $h$ of the proposed numerical scheme (S). This decomposition into random and non-random errors is analogous to the mean integrated squared error decomposition in the theory of non-parametric regression and image reconstruction~\cite{Korostelev1993}. Similarly to~\cite{Korostelev1993} we can use the bound in \eqref{eq:final-rate} to obtain rules of thumb on how to choose $k$ and $h$.  For example, we may first choose some value for $k$, and then choose $h$ so as to equate the two competing terms in \eqref{eq:final-rate}. This yields $h=k^{-\frac{1}{2(d+1)}}$ and \eqref{eq:final-rate} becomes
\begin{equation}\label{eq:final-rate2}
\|\hat{U}_h - U\|_{L^\infty([0,1]^d)} \leq Ck^{-\frac{1}{2d(d+1)}} = Ch^\frac{1}{d},
\end{equation}
with high probability.

Notice that Steps 1-3 in Algorithm \ref{alg:ndom}, i.e., computing $\hat{U}_h$, require $O(kh^{-d})$ operations.  If we choose the equalizing value $h=k^{-\frac{1}{2(d+1)}}$, then we find that computing $\hat{U}_h$ has complexity $O\left(k^\frac{3d+2}{2d+2}\right)$. Thus Algorithm \ref{alg:ndom} is sublinear in the following sense. Given $\eps>0$, we can choose $k$ large enough so that 
\[\|\hat{U}_h - U\|_{L^\infty([0,1]^d)} \leq \frac{\eps}{2c_d},\]
with high probability.  The $L^1$ sorting accuracy of using $\hat{U}_h$ in place of $u_n$ is then given by
\begin{align*}
\frac{1}{n} \sum_{i=1}^n |c_d \hat{U}_h(X_i) - dn^{-\frac{1}{d}} u_n(X_i)| &{}\leq{} \frac{1}{n} \sum_{i=1}^n \bigg(c_d|\hat{U}_h(X_i) - U(X_i)| \\
&\hspace{1in}+ |c_d U(X_i) - dn^{-\frac{1}{d}} u_n(X_i)|\bigg) \\
&{}\leq{} \frac{\eps}{2} + \frac{1}{n} \sum_{i=1}^n |c_d U(X_i) - dn^{-\frac{1}{d}} u_n(X_i)|,
\end{align*}
with high probability.   By the stochastic convergence $dn^{-\frac{1}{d}}u_n \to c_d U$, and the rates presented in Section \ref{sec:pde-rates}, there exists $N>0$ such that for all $n\geq N$ we have
\begin{equation}\label{eq:eps-approx}
\frac{1}{n} \sum_{i=1}^n |c_d \hat{U}_h(X_i) - dn^{-\frac{1}{d}} u_n(X_i)|  \leq \eps
\end{equation}
with high probability.
Thus, for any $\eps>0$ there exists $N,k$ and $h$ such that $\hat{U}_h$ is an $O(\eps)$ approximation of $u_n$ for all $n\geq N$, and $\hat{U}_h$ can be computed in constant time with respect to $n$.  We emphasize that the sublinear nature of the algorithm lies in the computation of $\hat{U}_h$.  Ranking all samples, i.e.,  evaluating $\hat{U}_h$ at each of $X_1,\dots,X_n$, and computing the $L^1$ error in \eqref{eq:eps-approx} of course requires $O(n)$ operations.  
In practice, it is often the case that one need not rank all $n$ samples (e.g., in a streaming application~\cite{gilbert2007}), and in such cases the entire algorithm is constant or sublinear in $n$ in the sense described above.

\subsection{Evaluation of Algorithm \ref{alg:ndom}}

We evaluated our proposed algorithm in dimension $d=2$ for a uniform density and a mixture of Gaussians given by $f(x) = \frac{1}{4}\sum_{i=1}^4 g_i(x)$,
where each $g_i:\R^2\to\R$ is a multivariate Gaussian density with covariance matrix $\Sigma_i$ and mean $\mu_i$.  We write the covariance matrix in the form $\Sigma_i=R_{\theta_i} {\rm diag}(\lambda_{i,1},\lambda_{i,2}) R^T_{\theta_i}$, where $R_{\theta}$ denotes a rotation  matrix, and $\lambda_{i,1}$, $\lambda_{i,2}$ are the eigenvalues.
The values for $\lambda_{i,j}, \mu_i$ and $\theta_i$ are given in Table \ref{tab:param}, and the density is illustrated in Figure \ref{fig:gauss-den}.
\begin{figure}[t!]
\begin{minipage}{\linewidth}
      \centering
      \begin{minipage}{0.4\linewidth}
         \begin{table}[H]
            \begin{tabular}{| l | l | l | l | l |}
               \hline
               & $\scriptstyle\lambda_{i,1}$ & $\scriptstyle\lambda_{i,2}$ & $\scriptstyle\theta_i$ & $\scriptstyle(\mu_{i,1},\mu_{i,2})$ \\ \hline
               $\scriptstyle g_1$ & $\scriptstyle 0.01$ & $\scriptstyle 0.00025$ & $\scriptstyle \frac{\pi}{3}$ & $\scriptstyle (0.2,0.5)$ \\\hline
               $\scriptstyle g_2$ & $\scriptstyle 0.0576$ & $\scriptstyle 0.00064$ & $\scriptstyle 0$ & $\scriptstyle (0.5,0.3)$ \\\hline
               $\scriptstyle g_3$ & $\scriptstyle 0.04$ & $\scriptstyle 0.00025$ & $\scriptstyle -\frac{\pi}{6}$ & $\scriptstyle (0.4,0.8)$ \\\hline
               $\scriptstyle g_4$ & $\scriptstyle 0.01$ & $\scriptstyle 0.01$ & $\scriptstyle 0$ & $\scriptstyle (0.8,0.8)$ \\
               \hline
           \end{tabular}
           \caption{Parameter values for mixture of Gaussians density}
           \label{tab:param}
        \end{table}
    \end{minipage}
    \hspace{0.05\linewidth}
      \begin{minipage}{0.5\linewidth}
         \begin{figure}[H]
         \centering
         \includegraphics[width=0.8\textwidth,clip=true,trim = 20 20 20 135]{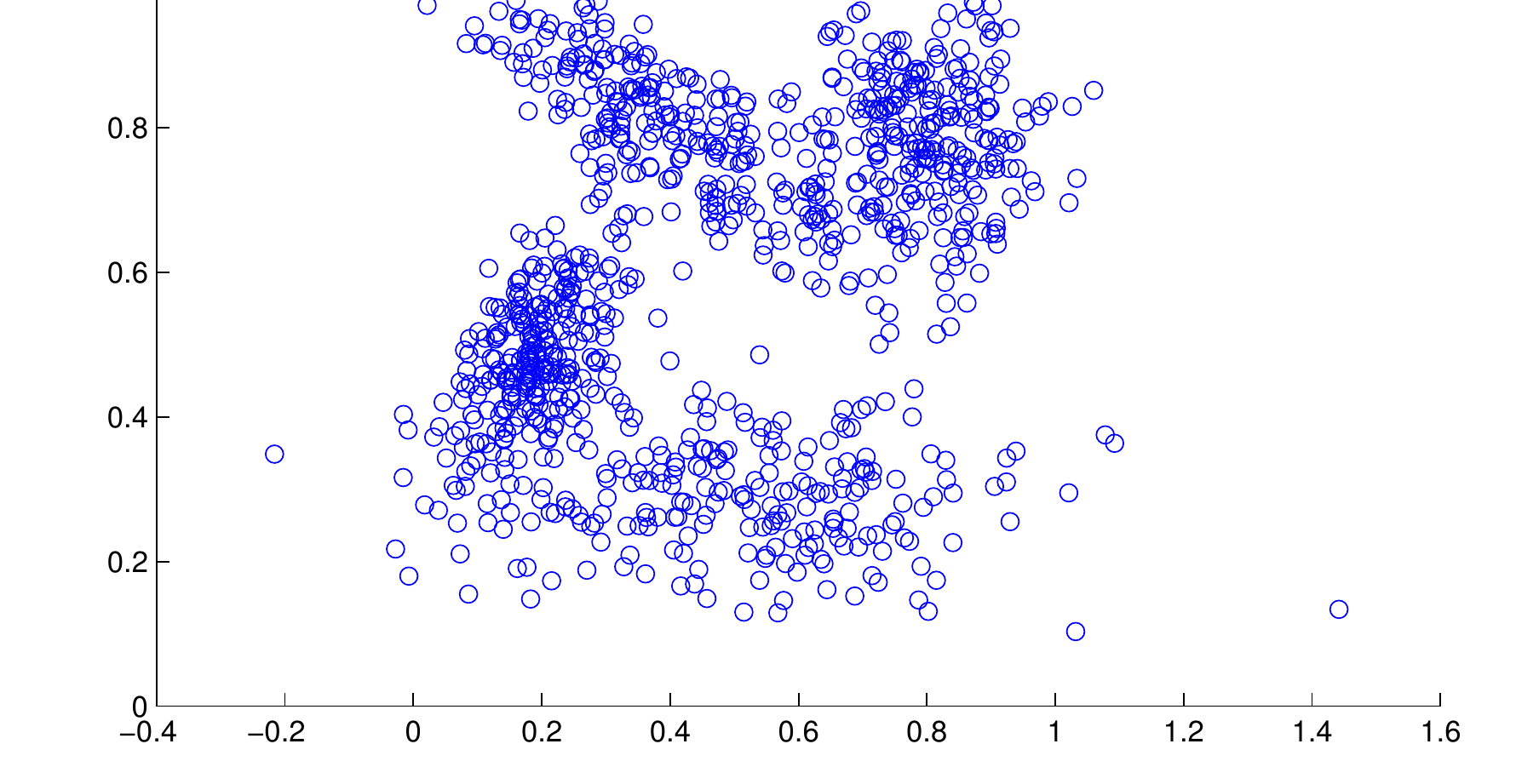}
         \caption{Depiction of random samples from the mixture of Gaussians density.}
         \label{fig:gauss-den}
         \end{figure}
     \end{minipage}
\end{minipage}
\end{figure}

It is important to evaluate the accuracy of the approximate sorting obtained by Algorithm \ref{alg:ndom}.  In practice, the numerical ranks assigned to each point are largely irrelevant, provided the relative orderings between samples are correct. Hence a natural accuracy measure for a given ranking is the fraction of pairs $(X_i,X_j)$ that are ordered correctly.  Recalling that the true Pareto rank is given by $u_n(X_i)$, this can be expressed as
\begin{equation}\label{eq:acc}
{\rm  Accuracy} = \frac{2}{n(n-1)} \sum_{i=1}^n \sum_{j=i+1}^n \psi(u_n(X_i) - u_n(X_j),\hat{U}_h(X_i) - \hat{U}_h(X_j)),
\end{equation}
where  $\psi(x,y) = 1$ if $xy > 0$ and $\psi(x,y)=0$ otherwise.
It turns out that the accuracy scores \eqref{eq:acc} for our algorithm are often very close to 1.  In order to make the plots easier to interpret visually, we have chosen to plot $-\log(1-\text{Accuracy})$ instead of Accuracy in \emph{all} plots.    

Unfortunately, the complexity of computing the accuracy score via \eqref{eq:acc} is $O(n^2)$, which is intractable for even moderate values of $n$.  We note however that \eqref{eq:acc} is, at least formally, a Monte-Carlo approximation of 
\[\int_{\R^d} \int_{\R^d} \psi(U(x) - U(y),U_h(x) - U_h(y))f(x)f(y) \, dxdy.\]
Hence it is natural to use a truncated Monte-Carlo approximation to estimate \eqref{eq:acc}.  This is done by selecting $n$ pairs $(X_{i_1},X_{j_1}),\dots,(X_{i_n},X_{j_n})$ at random and computing
\[\frac{1}{n} \sum_{k=1}^n \psi(u_n(X_{i_k}) - u_n(X_{j_k}),\hat{U}_h(X_{i_k}) - \hat{U}_h(X_{j_k})).\]
The complexity of the Monte-Carlo approximation is $O(n)$.  In all plots in the paper, we computed the Monte-Carlo approximation $10$ times and plotted means and error bars corresponding to a $95\%$ confidence interval.  In all of the figures, the confidence intervals are sufficiently small so that they are contained within the data point itself.

We can see in Figure \ref{fig:acc} that we can achieve excellent accuracy while maintaining a fixed grid and subsample size as a function of $n$.  We also see that, as expected, the accuracy increases when one uses more grid points for solving the PDE and/or more subsamples for estimating the density.  We also see that the algorithm works better on uniformly distributed samples than on the mixture of Gaussians.  Indeed, it is quite natural to expect the density estimation and numerical scheme to be less accurate when $f$ changes rapidly.  
\begin{figure}
\centering
\subfigure[Uniform distribution]{\includegraphics[width=0.45\textwidth,clip=true,trim = 0 0 0 20]{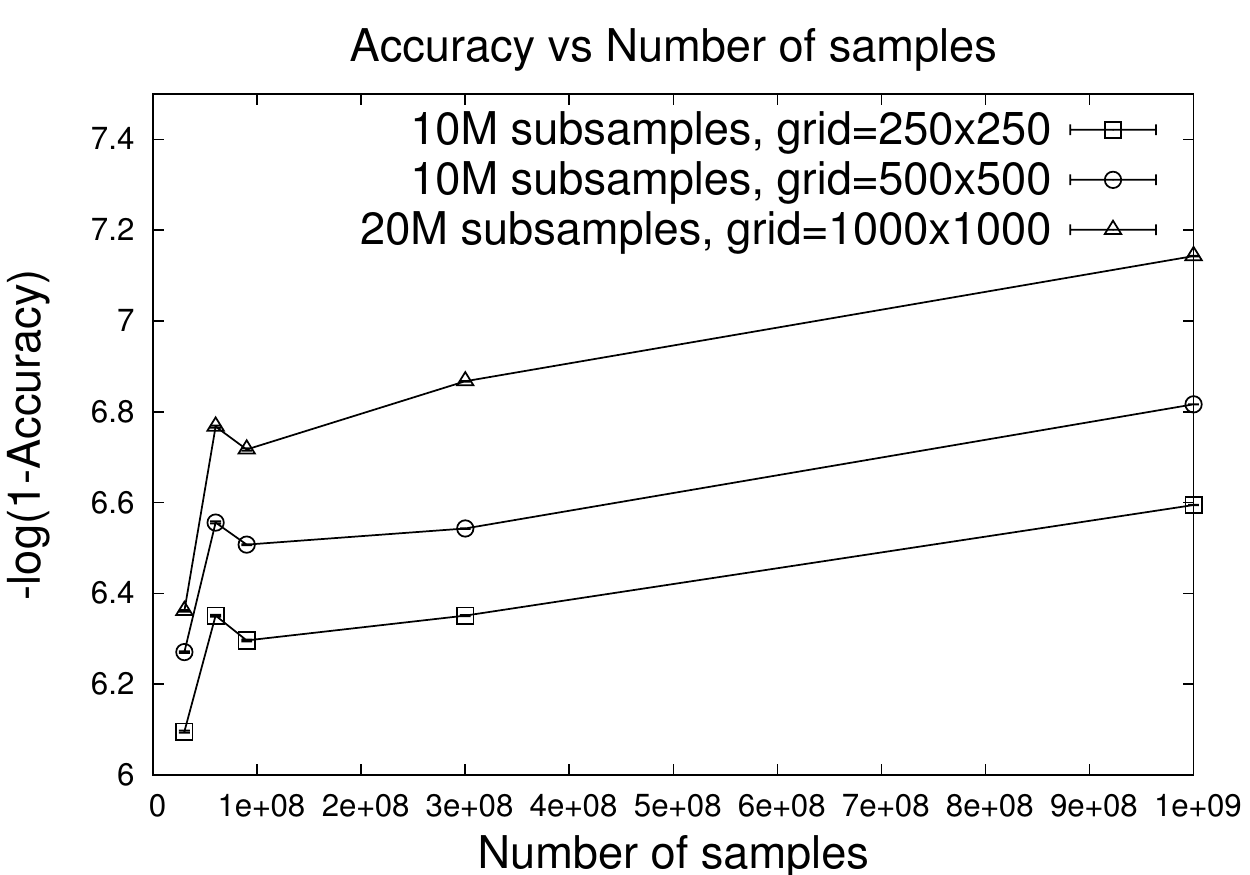}}
\subfigure[Mixture of Gaussians]{\includegraphics[width=0.45\textwidth,clip=true,trim = 0 0 0 20]{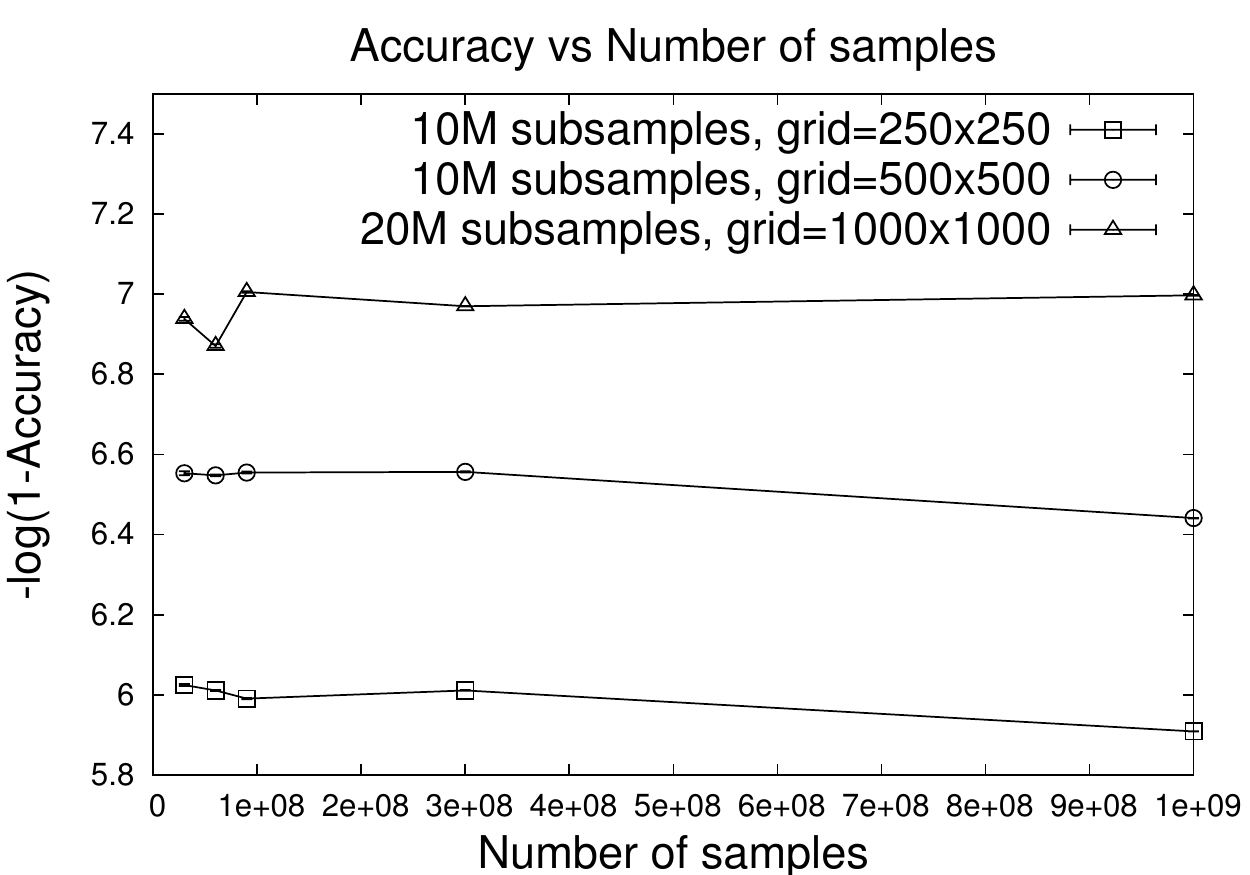}}
\caption{Comparison of accuracy versus number of samples for various grid sizes and number of subsamples $k$ used to estimate $f$.}
\label{fig:acc}
\end{figure}

We compared the performance of our algorithm against the fast two dimensional non-dominated sorting algorithm presented in~\cite{felsner1999}, which takes $O(n\log n)$ operations to sort $n$ points. The code for both algorithms was written in C++ and was compiled on the same architecture with the same compiler optimization flags.  Figure \ref{fig:cputime} shows a comparison of the CPU time used by each algorithm.  For our fast approximate sorting, we show the CPU time required to solve the PDE (Steps 1-3 in Algorithm \ref{alg:ndom}) separately from the CPU time required to execute all of Algorithm \ref{alg:ndom}, since the former is sublinear in $n$.
\begin{figure}
\centering
\subfigure[CPU time]{\includegraphics[width=0.45\textwidth,clip=true,trim = 0 0 0 20]{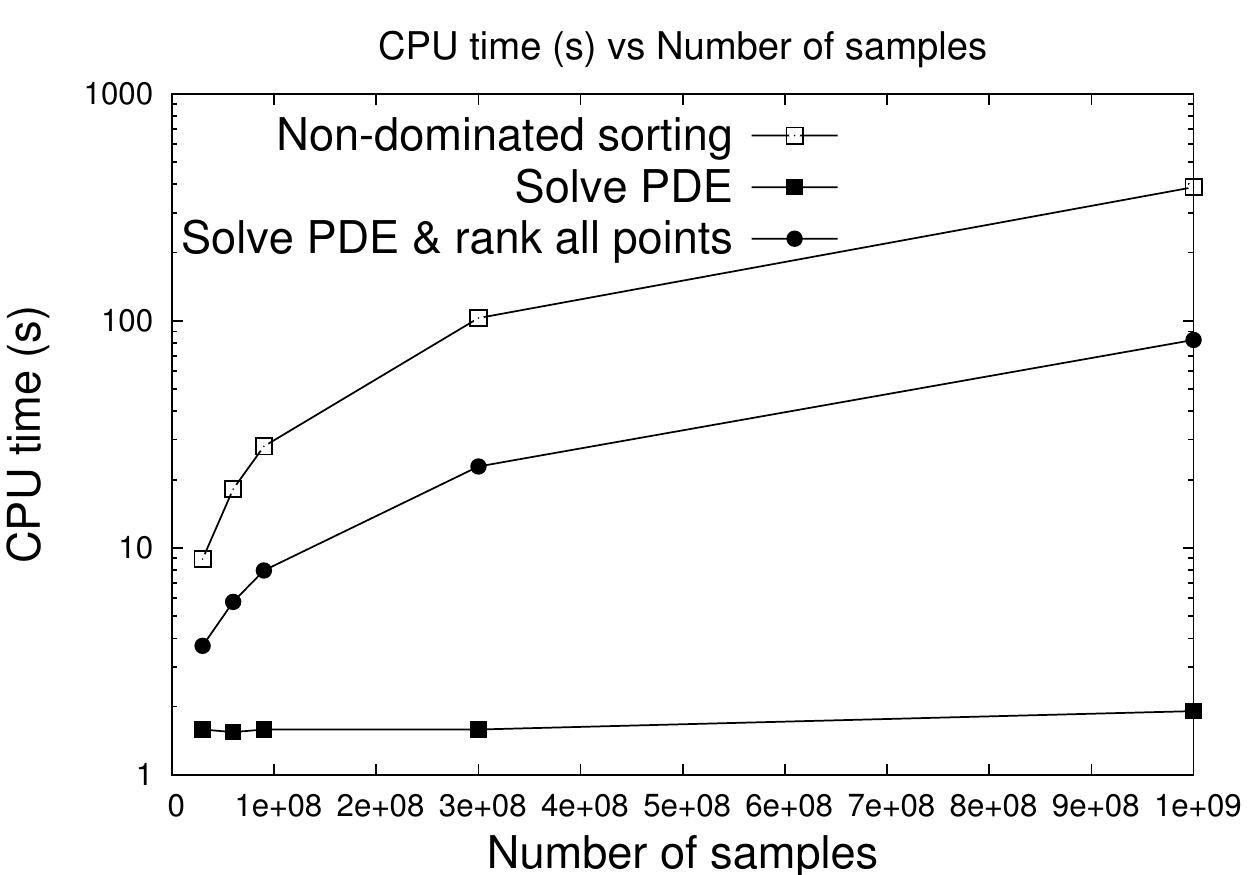}\label{fig:cputime}}
\subfigure[Accuracy vs grid size]{\includegraphics[width=0.45\textwidth,clip=true,trim = 0 0 0 20]{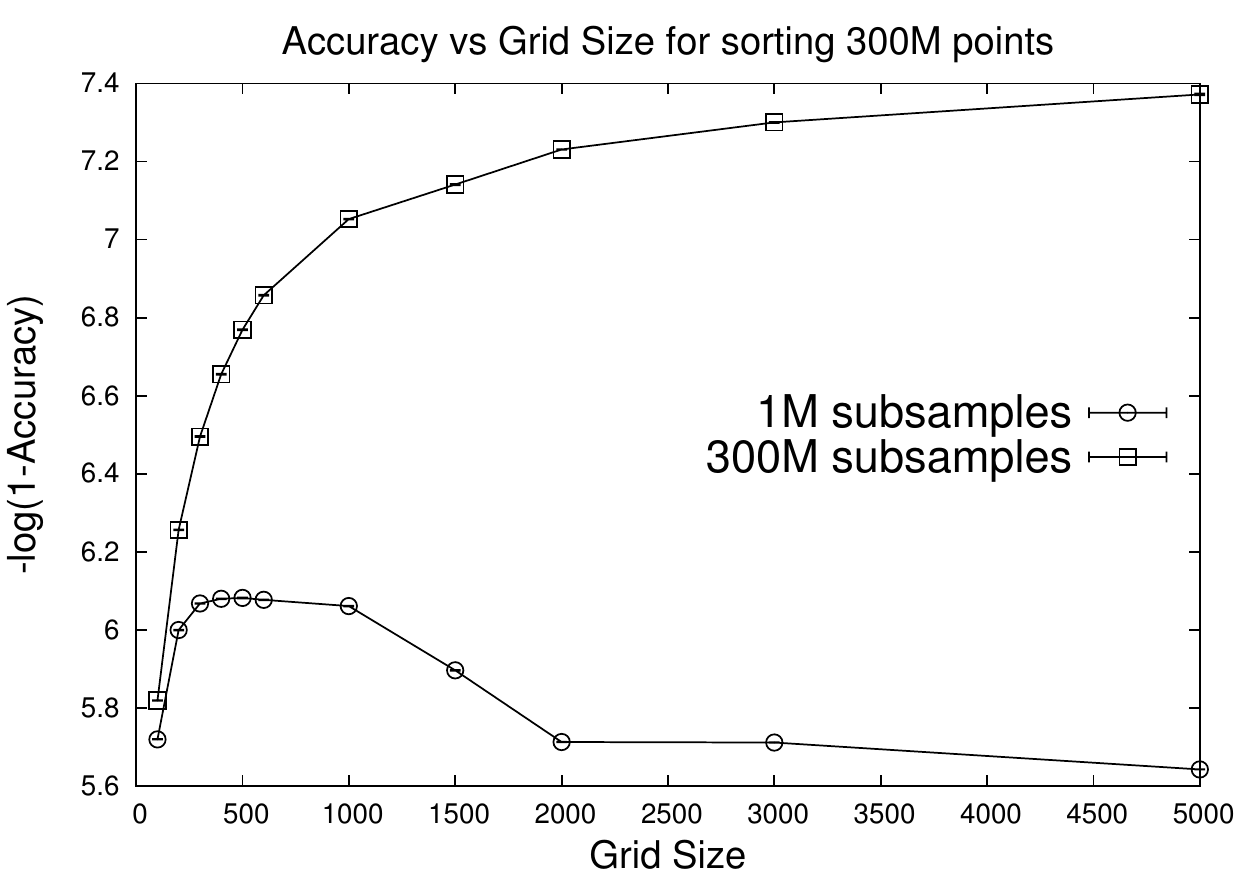}\label{fig:grid}}
\caption{(a) Comparison of CPU time versus number of samples for a grid size of $250\times 250$ and $k=10^7$ subsamples for estimating the density. (b) Comparison of accuracy versus grid size for $k=10^6$ and $k=3\times 10^8$ subsamples for non-dominated sorting of $n=3\times 10^8$ points. Notice that when $k$ is small compared to $n$ it is not always beneficial to use a finer grid for solving the PDE and estimating the density.}
\end{figure}

It is also interesting to consider the relationship between the grid size and the number of subsamples $k$. In Figure \ref{fig:grid}, we show accuracy versus grid size for $k=10^6$ and $k=3\times 10^8$ subsamples for non-dominated sorting of $n=3\times 10^8$ points.  Notice that for $k=10^6$ subsamples, it is not beneficial to use a finer grid than approximately $500 \times 500$.  This is quite natural in light of the error estimate on Algorithm \ref{alg:ndom}~\eqref{eq:final-rate}.


\subsection{Subset ranking}

There are certainly other ways one may think of to perform fast approximate sorting without invoking the PDE (P).  One natural idea would be to perform non-dominated sorting on a random subset of $X_1,\dots,X_n$, and then rank all $n$ points via some form of interpolation.  We will call such an algorithm \emph{subset ranking} (in contrast to the PDE-based ranking we have proposed).   Although such an approach is quite intuitive, it is important to note that there is, at present, no theoretical justification for such an approach.  Nonetheless, it is important to compare the performance of our algorithm against such an algorithm.  

Let us describe how one might implement a subset ranking algorithm.  As described above, the first step is to select a random subset of size $k$ from $X_1,\dots,X_n$. Let us call the subset $Y_1,\dots,Y_k$.  We then apply non-dominated sorting to $Y_1,\dots,Y_k$, which generates Pareto rankings $u_k(Y_i)$ for each $Y_i$.  The final step is to rank $X_1,\dots,X_n$ via interpolation.  There are many ways one might approach this.  In similar spirit to our PDE-based ranking (Algorithm 1), we use grid interpolation, using the same grid size as used to solve the PDE.  We compute a ranking at each grid point by averaging the ranks of all samples from $Y_1,\dots,Y_k$ that fall inside the corresponding grid cell.  The ranking of an arbitrary sample $X_i$ is then computed by linear interpolation using the ranks of neighboring grid points.  In this way, the rank of $X_i$ is an average of the ranks of nearby samples from $Y_1,\dots,Y_k$, and there is a grid size parameter which allows a meaningful comparison with PDE-based ranking (Algorithm 1).

Figure \ref{fig:subsample} shows the accuracy scores for PDE-based ranking (Algorithm 1) and subset ranking of $n=10^8$ samples drawn from the uniform and mixture of Gaussians distributions.  A grid size of $250\times 250$ was used for both algorithms, and we varied the number of subsamples from $k=10^3$ to $k=10^8$.  Notice a consistent accuracy improvement when using PDE-based ranking versus subset ranking, when the number of subsamples is significantly less than $n$.  It is somewhat surprising to note that subset ranking has much better than expected performance.   As mentioned previously, to our knowledge there is no theoretical justification for such a performance when $k$ is small.   

\begin{figure}
\centering
\subfigure[Uniform distribution]{\includegraphics[width=0.45\textwidth,clip=true,trim = 0 0 0 22]{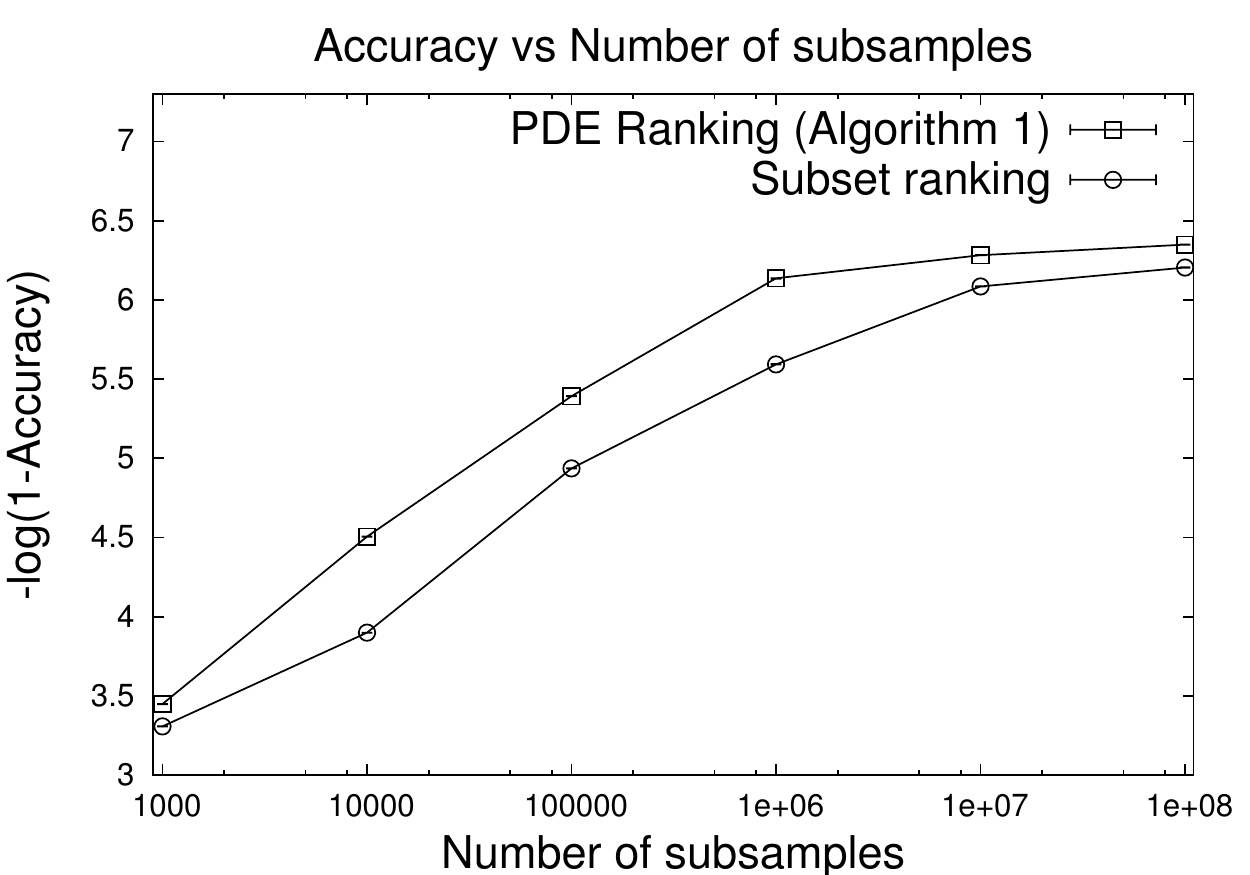}}
\subfigure[Mixture of Gaussians]{\includegraphics[width=0.45\textwidth,clip=true,trim = 0 0 0 22]{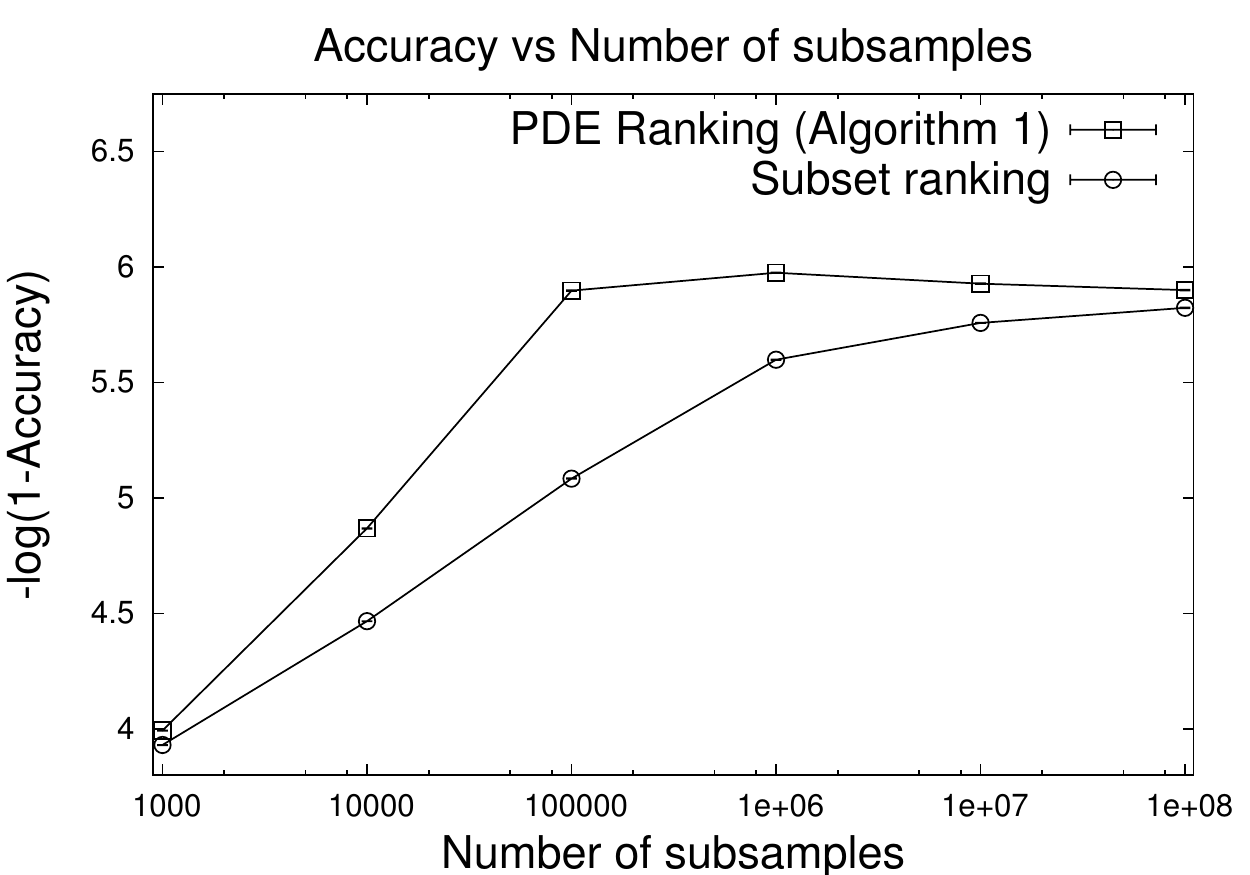}}
\caption{Comparison of PDE-based ranking (Algorithm 1) and naive subset interpolation ranking for sorting $n=10^8$ samples.  Accuracy scores are shown for various numbers of subsamples ranging from $k=10^3$ to $k=10^8$.}
\label{fig:subsample}
\end{figure}

\subsection{Application in anomaly detection}
\label{sec:real}

We now demonstrate Algorithm \ref{alg:ndom} on a large scale real data application of anomaly detection~\cite{hsiao2012}.  The data consists of thousands of pedestrian trajectories, captured from an overhead camera, and the goal is to differentiate nominal from anomalous pedestrian behavior in an unsupervised setting.  The data is part of the Edinburgh Informatics Forum Pedestrian Database and was captured in the main building of the School of Informatics at the University of Edinburgh~\cite{majecka2009}.  Figure \ref{fig:traj} shows 100 of the over 100,000 trajectories captured from the overhead camera.  

\begin{figure}
\centering
\subfigure[Example trajectories]{\includegraphics[clip=true, trim = 50 0 50 50, width=0.32\linewidth]{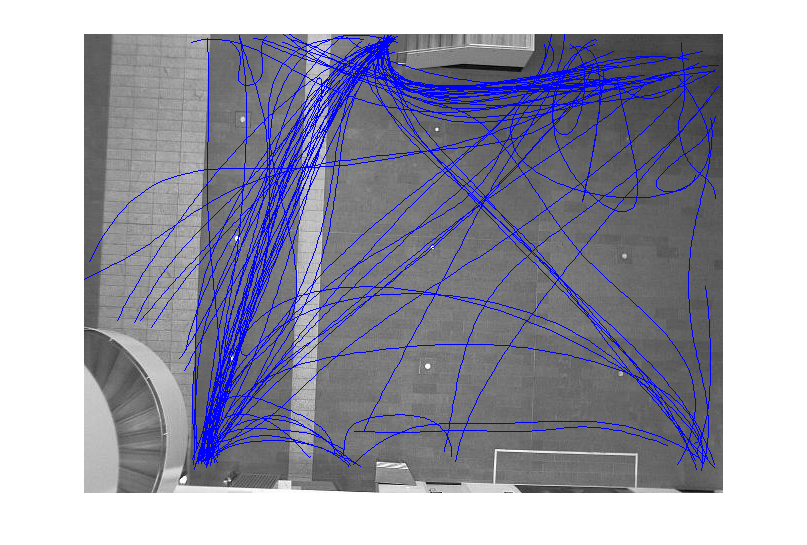}\label{fig:traj}}
\subfigure[50000 Pareto points]{\includegraphics[clip=true, trim = 10 10 20 25, width=0.32\linewidth]{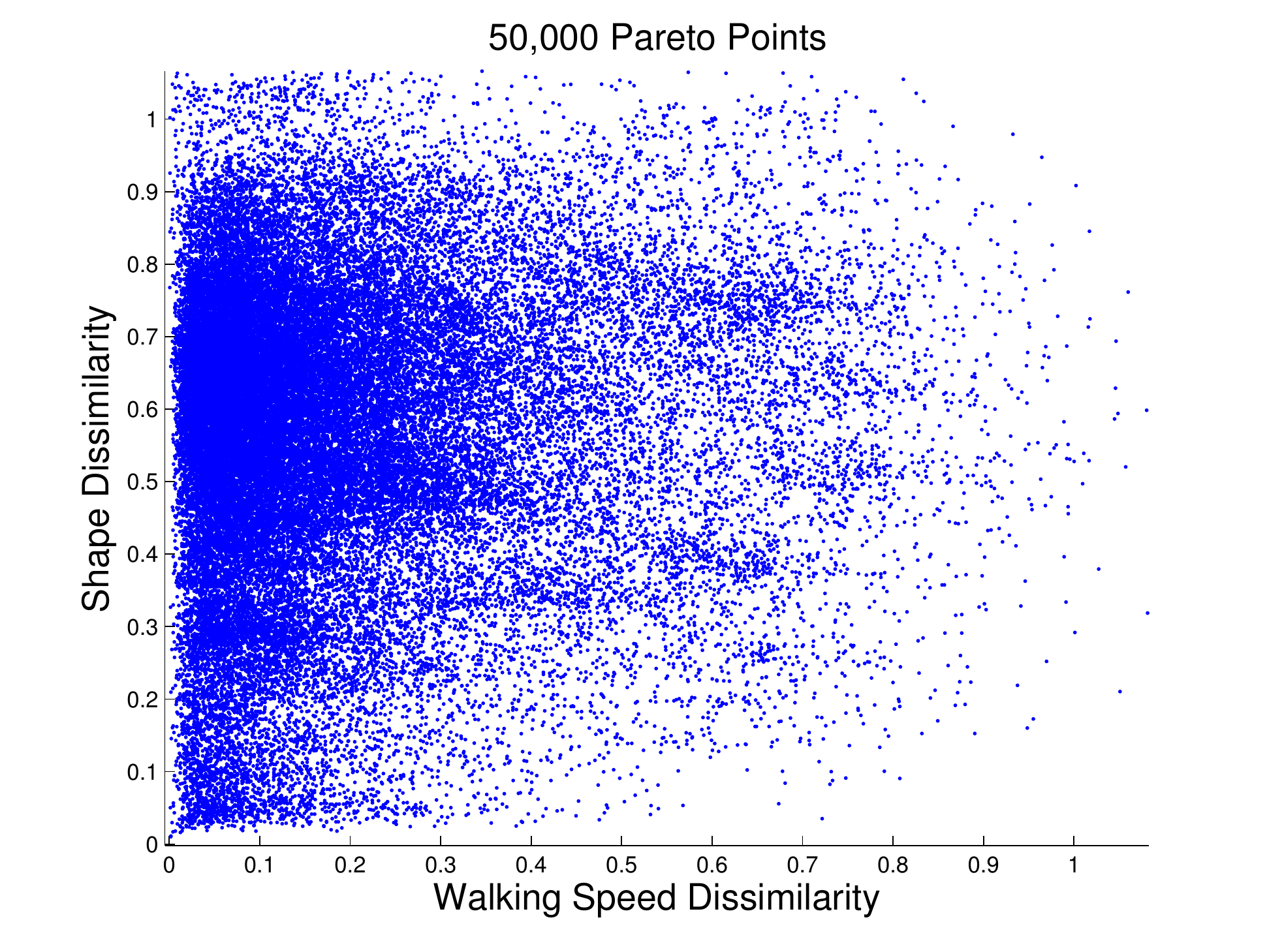}\label{fig:points}}
\subfigure[Pareto fronts]{\includegraphics[clip=true, trim = 10 10 20 10, width=0.32\linewidth]{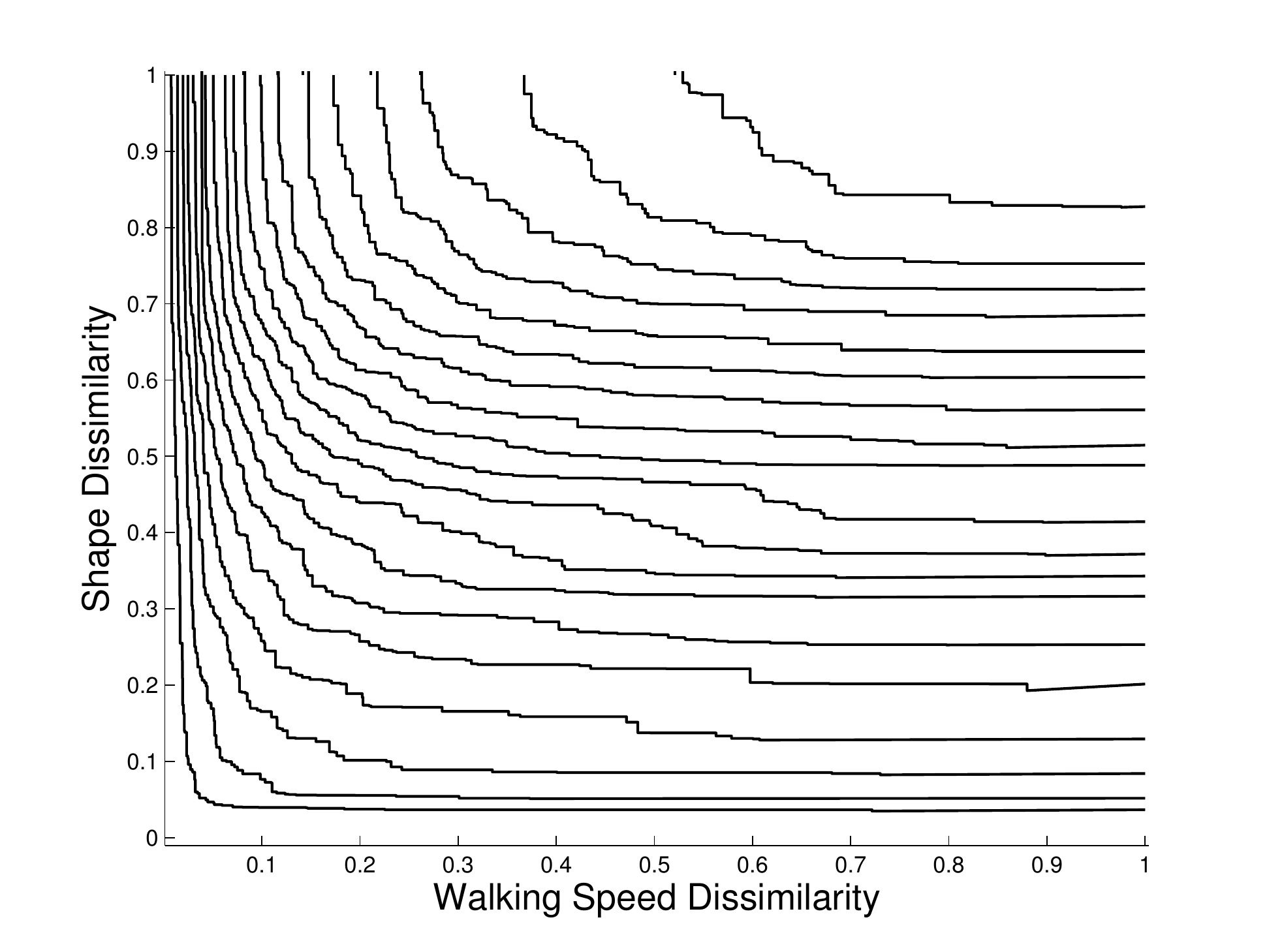}\label{fig:fronts}}
\caption{(a) Example pedestrian trajectories, (b) Plot of 50000 of the approximately $6\times 10^9$ Pareto points, (c) 30 evenly spaced Pareto fronts computed from the 50000 points in (b).}
\end{figure}

The approach to anomaly detection employed in~\cite{hsiao2012} utilizes multiple criteria to measure the dissimilarity between trajectories, and combines the information using a Pareto-front method, and in particular, non-dominated sorting. The database consists of a collection of trajectories $\{\gamma_1,\dots,\gamma_M\}$, where $M=110035$, and the criteria used in~\cite{hsiao2012} are a walking speed dissimilarity, and a trajectory shape dissimilarity.  Given two trajectories $\gamma_i,\gamma_j:[0,1] \to [0,1]^2$, the walking speed dissimilarity $c_{speed}(\gamma_i,\gamma_j)$ is the $L^2$ distance between velocity histograms of each trajectory, and the trajectory shape dissimilarity is the $L^2$ distance between the trajectories themselves, i.e., $c_{shape}(\gamma_i,\gamma_j) = \|\gamma_i-\gamma_j\|_{L^2(0,1)}$.  There is then a Pareto point $X_{i,j} = (c_{speed}(\gamma_i,\gamma_j),c_{shape}(\gamma_i,\gamma_j))$ for every pair of trajectories $(\gamma_i,\gamma_j)$, yielding $\binom{M}{2} \approx 6\times 10^9$ Pareto points.  Figure \ref{fig:points} shows an example of 50000 Pareto points and Figure \ref{fig:fronts} shows the respective Pareto fronts.  In~\cite{hsiao2012}, only 1666 trajectories from one day were used, due to the computational complexity of computing the dissimilarities and non-dominated sorting.

The anomaly detection algorithm from~\cite{hsiao2012} performs non-dominated sorting on the Pareto points $\{X_{i,j}\}_{1 \leq i< j \leq M}$, and uses this sorting to define an anomaly score for every trajectory $\gamma_i$.  Let $n=\binom{M}{2}$ and let $u_n:\R^2\to \R$ denote the longest chain function corresponding to this non-dominated sorting.  The anomaly score for a particular trajectory $\gamma_i$ is defined as
\[s_i = \frac{1}{M}\sum_{j=1}^M u_n(c_{speed}(\gamma_i,\gamma_j),c_{shape}(\gamma_i,\gamma_j)),\]
and trajectories with an anomaly score higher than a predefined threshold $\sigma$ are deemed anomalous. 

Using Algorithm \ref{alg:ndom}, we can approximate $u_n$ using only a small fraction of the Pareto points $\{X_{i,j}\}_{1 \leq 1 < j\leq M}$, thus alleviating the computational burden of computing all pairwise dissimilarities.  Figure \ref{fig:acc-real} shows the accuracy scores for Algorithm \ref{alg:ndom} and subset ranking versus the number of subsamples $k$ used in each algorithm. Due to the memory requirements for non-dominated sorting, we cannot sort datasets significantly larger than than $10^9$ points. Although there is no such limitation on Algorithm 1, it is important to have a ground truth sorting to compare against.   Therefore we have used only $44722$ out of $110035$ trajectories, yielding approximately $10^9$ Pareto points.   For both algorithms, a $500\times500$ grid was used for solving the PDE and interpolation.  Notice the accuracy scores are similar to those obtained for the test data in Figure \ref{fig:acc}.  This is an intriguing observation in light of the fact that $\{X_{i,j}\}_{1\leq i < j \leq M}$ are \emph{not} \iid, since they are elements of a Euclidean dissimilarity matrix.

\begin{figure}
\centering
\includegraphics[width=0.43\textwidth,clip=true,trim = 0 0 0 0]{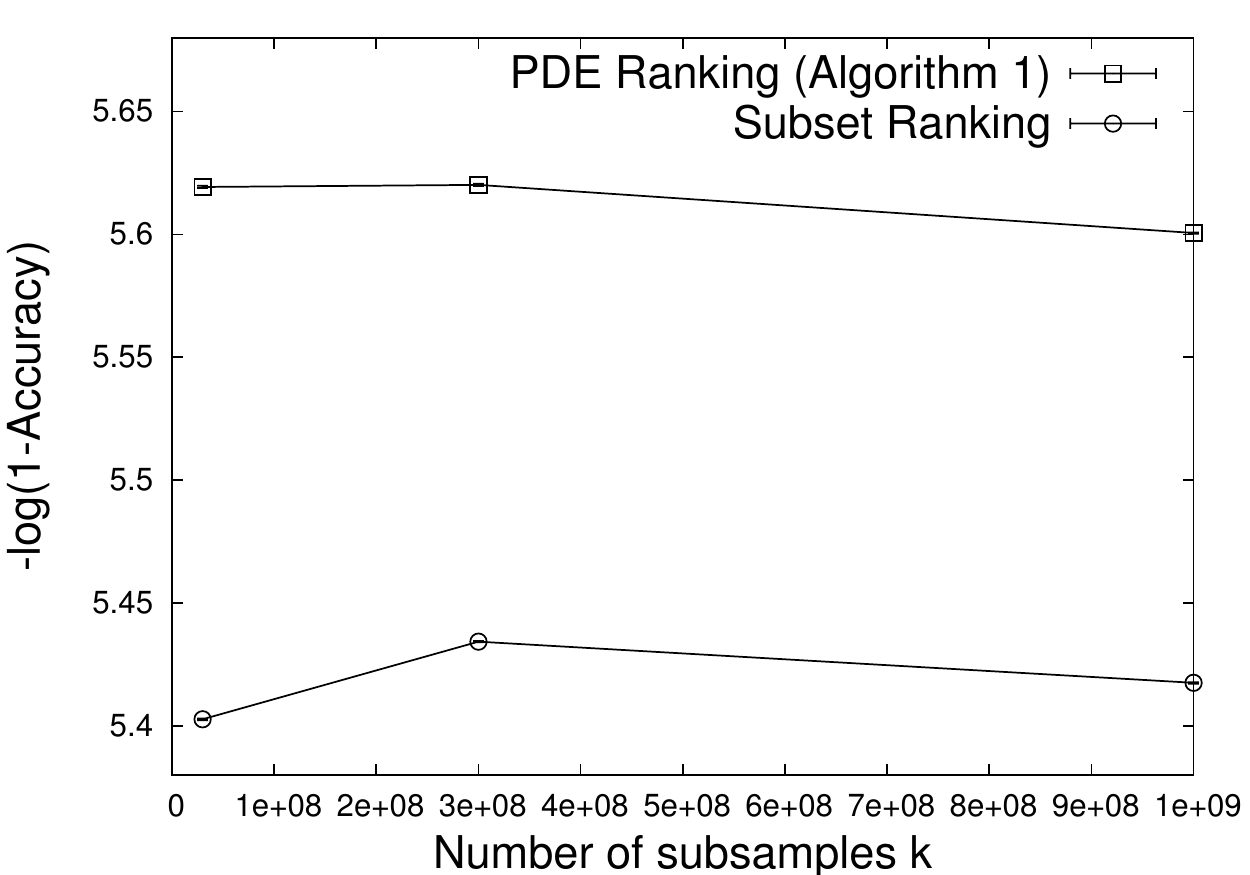}
\caption{Accuracy scores for Algorithm \ref{alg:ndom} and subset ranking for sorting $10^9$ Pareto points from the pedestrian anomaly detection problem versus the number of subsamples $k$.}\label{fig:acc-real}
\end{figure}

\subsection{Discussion}
\label{sec:disc}

We have provided theory that demonstrates that, when $X_1,\dots,X_n$ are \iid~in $\R^2$ with a nicely behaved density function $f$, the numerical scheme (S) for  (P) can be utilized to perform fast approximate non-dominated sorting with a high degree of accuracy. We have also shown that in a real world example with non-\iid~data, the scheme (S) still obtains excellent sorting accuracy.    We expect the same algorithm to be useful in dimensions $d=3$ and $d=4$ as well, but of course the complexity of solving (P) on a grid increases exponentially fast in $d$.  In higher dimensions, one could explore other numerical techniques for solving (P) which do not utilize a fixed grid~\cite{cecil2004}. At present, there is also no good algorithm for non-dominated sorting in high dimensions.  The fastest known algorithm is $O(n(\log n)^{d-1})$~\cite{jensen2003}, which becomes intractable when $n$ and $d$ are large.  

This algorithm has the potential to be particularly useful in the context of big data streaming problems~\cite{gilbert2007}, where it would be important to be able to construct an approximation of the Pareto depth function $u_n$ without visiting all the datapoints $X_1,\dots,X_n$, as they may be arriving in a data stream and it may be impossible to keep a history of all samples.  In such a setting, one could slightly modify Algorithm \ref{alg:ndom} so that upon receiving a new sample, the estimate $\hat{f}_h$ is updated, and every so often the scheme (S) is applied to recompute the estimate of $\hat{U}_h$. 

There are certainly many situations in practice where the samples $X_1,\dots,X_n$ are not \iid, or the density $f$ is not nicely behaved.  In these cases, there is no reason to expect our algorithm to have much success, and hence we make no claim of universal applicability.  However, there are many cases of practical interest where these assumptions are valid, and hence this algorithm can be used to perform fast non-dominated sorting in these cases. Furthermore, as we have demonstrated in Section \ref{sec:real}, there are situations in practice where the \iid~assumption is violated, yet our proposed algorithm maintains excellent accuracy and performance.

We proposed a simple \emph{subset ranking} algorithm based on sorting a small subset of size $k$ and then performing interpolation to rank all $n$ samples.  Although there is currently no theoretical basis for such an algorithm, we showed that subset ranking achieves surprisingly high accuracy scores and is only narrowly outperformed by our proposed PDE-based ranking.  The simplicity of subset ranking makes it particularly appealing, but more research is needed to prove that it will always achieve such high accuracy scores for moderate values of $k$.

We should note that there are many obvious ways to improve our algorithm.  Histogram approximation to probability densities is quite literally the most basic density estimation algorithm, and one would expect to obtain better results with more sophisticated estimators.   It would also be natural to perform some sort of histogram equalization to $X_1,\dots,X_n$ prior to applying our algorithm in order to spread the samples out more uniformly and smooth out the effective density $f$.  Provided such a transformation preserves the partial order $\leqq$ it would not affect the non-dominated sorting of $X_1,\dots,X_n$.  In the case that $f$ is separable (a product density), one can perform histogram equalization on each coordinate independently to obtain uniformly distributed samples. We leave these and other potential improvements to future work; our purpose in this paper has been to demonstrate that one can obtain excellent results with a very basic algorithm.

\section*{Acknowledgments}

We thank Ko-Jen Hsiao for providing code for manipulating the pedestrian trajectory database.  

\appendix
\section*{Appendix}
We use the following minor extension of the Arzel\`a-Ascoli Theorem in Section \ref{sec:convergence-proof}.
Let $X$ be a compact metric space.  We say that a sequence $\{f_n\}_{n=1}^\infty$ of real-valued functions on $X$ is \emph{approximately equicontinuous} if for every $\eps>0$ there exists $\delta > 0$ such that
\begin{equation}\label{eq:def}
\forall x,y \in X, \ |x-y| < \delta \implies |f_n(x) - f_n(y)| < \eps + \frac{1}{n},
\end{equation}
for every $n \in \N$.  
\begin{theorem}
Let $\{f_n\}_{n=1}^\infty$ be approximately equicontinuous and uniformly bounded. Then there exists a subsequence of $\{f_n\}_{n=1}^\infty$ converging uniformly on $X$ to a continuous function $f:X\to \R$. 
\end{theorem}
\begin{proof}
Let $\{x_i\}_{i=1}^\infty$ be a countably dense set in $X$.  By a Cantor diagonal argument, we can extract a subsequence $\{f_{n_k}\}_{k=1}^\infty$ such that for all $i \in \N$, $\{f_{n_k}(x_i)\}_{k=1}^\infty$ is a convergent sequence.

Let $\eps > 0$.  
Since $\{f_n\}_{n=1}^\infty$ is approximately equicontinuous there exists $\delta>0$ such that for all $n$ we have
\begin{equation}\label{eq:b}
|f_n(x) - f_n(y) | < \frac{\eps}{4} + \frac{1}{n} \ \ \text{for all } x,y \in X \text{ with } |x-y| < \delta.
\end{equation}
The collection of open balls $\{B_{\delta/2}(z)\}_{z \in X}$ forms an open cover of $X$.  Since $X$ is compact, there exists a finite subcover $B_1,\dots,B_M$ for some integer $M$.
Without loss of generality we may assume that $x_i \in B_i$.  Now let $x \in X$.  By \eqref{eq:b} we have
\begin{align*}
|f_{n_k}(x) - f_{n_j}(x)|  &{}\leq{} |f_{n_k}(x) - f_{n_k}(x_i)| + |f_{n_k}(x_i) - f_{n_j}(x_i)| + |f_{n_j}(x_i) - f_{n_j}(x)| \\
&{}<{} \frac{\eps}{2} + \frac{1}{n_k} + \frac{1}{n_j}+|f_{n_k}(x_i) - f_{n_j}(x_i)|,
\end{align*}
for some $i \in \{1,M\}$ and any $k,j$.
Hence we have
\[\|f_{n_k} - f_{n_j}\|_{L^\infty(X)} \leq  \frac{\eps}{2} + \frac{1}{n_k} + \frac{1}{n_j}+\sup_{1\leq i\leq M}|f_{n_k}(x_i) - f_{n_j}(x_i)|.\]
It follows that $\{f_{n_k}\}_{k=1}^\infty$ is Cauchy in $L^\infty$, which completes the proof.
\end{proof}

\end{document}